\documentclass[11pt,a4paper]{amsart}
\usepackage{mathtools}
\usepackage{amssymb}
\usepackage{dsfont,amsfonts}
\usepackage{url}
\usepackage{epsfig}
\usepackage{epstopdf}
\usepackage{url}
\usepackage[left=25mm,top=25mm,bottom=25mm,right=25mm]{geometry}
\usepackage[english]{babel} 
\usepackage[utf8]{inputenc} 
\usepackage{cmap}           
\usepackage{hyperxmp}       
\usepackage{enumerate}
\usepackage{mleftright}
\usepackage{array}
\usepackage{eqnarray}
\usepackage{fancybox}
\usepackage[pdfdisplaydoctitle = true,
      colorlinks = true,
      urlcolor = blue,
      citecolor = blue,
      linkcolor = blue,
      pdfstartview = FitH,
      pdfpagemode = UseNone,
      bookmarksnumbered = true,
      unicode=true]{hyperref} 
\usepackage{todonotes}
\usepackage{subcaption}
\usepackage{xcolor}
\usepackage{changes}
\usepackage{booktabs}
\usepackage{float}
\usepackage{xfrac}
\usepackage{mathrsfs}

\graphicspath{{figs/}}
\newtheorem{thm}{Theorem}[section]

\newtheorem{lem}[thm]{Lemma}

\theoremstyle{definition}
\newtheorem{defn}[thm]{Definition}

\newtheorem{rem}[thm]{Remark}
\newtheorem{exmp}[thm]{Example}

\newcommand{\RR}{\mathbb{R}}                
\newcommand{\ZZ}{\mathbb{Z}}                

\newcommand{\Ela}{\mathbb{E}\mathrm{la}}    
\newcommand{\Piez}{\mathbb{P}\mathrm{iez}}  
\newcommand{\Sym}{\mathbb{S}}               
\newcommand{\J}{\mathcal{J}}                

\newcommand{\OO}{\mathrm{O}}                
\newcommand{\SO}{\mathrm{SO}}               
\newcommand{\octa}{\mathbb{O}}              
\newcommand{\ico}{\mathbb{I}}               
\newcommand{\tetra}{\mathbb{T}}             
\newcommand{\DD}{\mathbb{D}}                
\newcommand{\1}{\mathds{1}}		            
\newcommand{\id}{\mathrm{Id}}                

\newcommand{\ee}{\pmb{e}}                   
\newcommand{\nn}{\pmb{n}}                   
\newcommand{\uu}{\pmb{u}}                   
\newcommand{\vv}{\pmb{v}}                   

\newcommand{\vR}{\mathbf{r}}
\newcommand{\bs}{\mathbf{s}}

\newcommand{\bS}{\mathbf{S}}

\newcommand{\bE}{\mathbf{E}}                
\newcommand{\bP}{\mathbf{P}}                
\newcommand{\bsigma}{\boldsymbol{\sigma}}       
\newcommand{\beps}{\boldsymbol{\varepsilon}}
\newcommand{\bd}{\mathbf{d}}
\newcommand{\be}{\mathbf{e}}
\newcommand{\ba}{\mathbf{a}}
\newcommand{\bb}{\mathbf{b}}

\newcommand{\set}[1]{\left\{#1\right\}}     

\newcommand{\Dnz}{\DD_n^z}
\newcommand{\Dnd}{\DD_{2n}^d}

\definechangesauthor[name=MO, color=red]{MO}

\hypersetup{
  pdfauthor={Perla Azzi and Marc Olive}, 
  pdftitle={Clips operation between Type II and type III $\OO(3)$-subgroups with application to Piezoelectricity law}, 
  pdfsubject={MSC 2020: 58D19 (20C35, 74F15)}, 
  pdfkeywords={Symmetry classes, Piezoelectricity tensor, Clips operation}, 
  pdflang=en, 
  }
\begin{document}

\title[Clips operation between $\OO(3)$-subgroups with application to Piezoelectricity]{Clips operation between Type II and type III $\OO(3)$-subgroups with application to Piezoelectricity}%

\author{P. Azzi}
\address[Perla Azzi]{CNRS, Sorbonne université, IMJ - Institut de mathématiques Jussieu, 75005, Paris, France \& Université Paris-Saclay, CentraleSupélec, ENS Paris-Saclay, CNRS, LMPS - Laboratoire de Mécanique Paris-Saclay, 91190, Gif-sur-Yvette, France}
\email[Perla Azzi]{perla.azzi@imj-prg.fr}

\author{M. Olive}
\address[Marc Olive]{Université Paris-Saclay, CentraleSupélec, ENS Paris-Saclay, CNRS, LMPS - Laboratoire de Mécanique Paris-Saclay, 91190, Gif-sur-Yvette, France}
\email{marc.olive@math.cnrs.fr}


\date{\today}%

\begin{abstract}
  Clips is a binary operation between conjugacy classes of closed subgroups of a compact group, which has been introduced to obtain the isotropy classes of a direct sum of representations $E_{1}\oplus E_{2}$, if the isotropy classes are known for each factor $E_{1}$, $E_{2}$. This allows, in particular, to compute the isotropy classes of any reducible representation, once one knows the decomposition into irreducible factors and the symmetry classes for the irreducible factors. In the specific case of the three dimensional orthogonal group $\OO(3)$, clips between some types of subgroups of $\OO(3)$ have already been calculated. However, until now, clips between type II and type III subgroups were missing. Those are encountered in 3D coupled constitutive laws. In this paper, we complete the clips tables by computing the missing clips. As an application, we obtain 25 isotropy classes for the standard $\OO(3)$ representation on the full 3D Piezoelectric law, which involves the three Elasticity, Piezoelectricity and Permittivity constitutive tensors.
\end{abstract}

\maketitle


\begin{scriptsize}
  \setcounter{tocdepth}{2}
  \tableofcontents
\end{scriptsize}

\section{Introduction}
\label{sec:intro}

Given a linear representation $(V,\rho)$ of a compact group $G$, the \emph{symmetry classes} (or isotropy classes) are defined as the sets of conjugacy classes of the symmetry subgroups of $G$ (see \cite{Bredon1960}, for more details). Finding explicitly the isotropy classes of a given representation has always been a difficult task, but it is known that there exists only a finite number of them. This fact was first conjectured by Montgomery \cite[problem 45]{eilenberg1949} and proved by Mostow in the case of the action of a compact Lie group on a compact manifold \cite{Mostow1957} using the result of Floyd \cite{floyd1957}.

The explicit calculation of isotropy classes is an interesting problem not only in mathematics but also in mechanics because the symmetry classes are connected to material symmetries. The problem of classifying materials according to their symmetries goes back to the work of Lord Kelvin \cite{kelvin1890}. Since then, many authors devoted a great effort to formulate the problem especially for $\SO(3)$ and $\OO(3)$ tensorial representations (used to model constitutive laws in mechanics). Surprisingly, it was only in 1996 that Forte and Vianello~\cite{forte1996} solved definitively the problem for the \emph{Elasticity tensor} (a fourth-order tensor). In that case, eight isotropy classes were found. Based on this method, the problem was solved for other constitutive laws. For instance, 16 isotropy classes were obtained for the \emph{Piezoelectricity tensor} (a third-order tensor)~\cite{Nye1985,ZB1994,Wel2004,New2005,ZTP2013}. Similar results were obtained for other constitutive tensor spaces \cite{forte1997symmetry,LHQ2011}.

However, the Forte-Vianello approach requires rather fine calculations and reasoning to establish the classification. This complexity makes difficult its application to more involved situations, such as constitutive tensors of order greater than 4 or coupled constitutive laws involving a family of tensors~\cite{Jur1974,EM1990}. A systematic way to calculate the isotropy classes was proposed by Chossat and Guyard in~\cite{Chossat1994} using a binary operation between conjugacy classes. This operation was named the \emph{clips} operation in~\cite{Olive2013,Olive2014,Olive2019,Olive2021}, where it was generalized and used to determine isotropy classes for reducible representations.

Clips tables for $\SO(3)$-subgroups were obtained first in~\cite{Chossat1994}. The problem is more complicated for $\OO(3)$-subgroups, since there exist three types of subgroups (see \autoref{sec:O3-subgroups}), where type I corresponds to $\SO(3)$-subgroups. Clips tables for $\OO(3)$-subgroups were calculated in~\cite{Olive2019}, except the ones between type II and type III subgroups. In this paper, we complete these results by providing these missing tables, which have never been obtained before. The clips tables
\ref{tab:Clips} and \ref{tab:Clips2} given in this paper, together with the tables provided in~\cite{Olive2019}, furnish an exhaustive list of clips tables between $\OO(3)$-subgroups. Since the isotropy classes are known for the irreducible representations of $\OO(3)$~\cite{GSS1988,CLM1990}, this allows, in practice, to determine the isotropy classes of any reducible representation of $\OO(3)$. As an original application, we used these results to determine the 25 isotropy classes of the full 3D Piezoelectricity law, in which three constitutive tensors are involved: the Elasticity, the Piezoelectricity and the Permittivity tensors.

\subsection*{Organization of the paper}

The paper is organized as follows. In \autoref{sec:Iso_classes}, we recall basic material on symmetry classes and we introduce the \emph{clips operation}. In \autoref{sec:main_result}, we provide (in tables \ref{tab:Clips} and \ref{tab:Clips2}) the list of clips between $\OO(3)$-subgroups of type II and III. The proofs and calculations for these clips are given in \autoref{sec:proof_clips}. The application to the full Piezoelectricity law is detailed in \autoref{sec:piezoelectricity}, where 25 symmetry classes are found. To be self-contained, the classification of closed $\OO(3)$-subgroups, up to conjugacy, is recalled in \autoref{sec:O3-subgroups}.

\section{Isotropy classes and Clips operation}
\label{sec:Iso_classes}

In this section, we recall the notions of isotropy groups and isotropy classes of a group representation and we introduce clips operation for sets of conjugacy classes. We consider a linear representation $(V,G,\rho)$ of a group $G$ on a finite dimensional vector space $V$, \textit{i.e.}, a group morphism
\begin{equation*}
  \rho:\ G\ \rightarrow\ \mathrm{GL}(V)
\end{equation*}
with $\mathrm{GL}(V)$, the group of invertible linear mappings of $V$ into itself.

The \emph{isotropy (or symmetry) group} of $\vv\in V$ is defined as
\begin{equation*}
  G_{\vv} := \set{g\in G; \; \rho(g)\vv=\vv},
\end{equation*}
and the \emph{isotropy (or symmetry) class} of $\vv$ is the conjugacy class $[G_{\vv}]$ of its isotropy group
\begin{equation*}
  [G_{\vv}] := \set{gG_{\vv}g^{-1};\; g\in G}.
\end{equation*}

For a given subgroup $H\subset G$, the conjugacy class $[H]$ is an \emph{isotropy class} of the linear representation $(V,G,\rho)$ if $H$ is conjugate to some isotropy group $G_{\vv}$. Let us then define $\J(V)$ to be the set of all isotropy classes of the representation $V$
\begin{equation*}
  \J(V) := \set{[G_{\vv}];\; \vv\in V}.
\end{equation*}

As mentioned in the introduction, the finiteness of $\J(V)$ is proven when $G$ is a compact Lie group and $V$ is a compact manifold (see, for instance, \cite{Mostow1957}, \cite{Bredon1960}, \cite{Mann1962}). The result for a linear representation of a compact Lie group on a vector space follows immediately, since, then, an invariant scalar product exists and the restriction of the representation to the unit sphere, which is a compact manifold, has the same isotropy classes as the full representation.

\begin{thm}
  Let $(V,G)$ be a linear representation of a compact Lie group $G$ on a vector space $V$. Then there exists a finite number of isotropy classes
  \begin{equation*}
    \J(V)=\set{[H_1],\dotsc, [H_n]}.
  \end{equation*}
\end{thm}

\begin{rem}
  The set of conjugacy classes $[H]$ of closed subgroups of a compact group is endowed with a partial order relation (see~\cite{bredon1972introduction}), given by
  \begin{equation*}
    [H]\preceq [K]\iff \exists g\in G,\quad gHg^{-1}\subset K.
  \end{equation*}
\end{rem}
Given this finiteness result, it is important to explicitly calculate the isotropy classes of a given representation. In the specific case of $\SO(3)$ linear representations, Michel~\cite{Michel1980} obtained the isotropy classes for the \emph{irreducible representations} and the ones for $\OO(3)$ irreducible representations were obtained by Ihrig and Golubitsky~\cite{Ihrig1984}.
Thereafter, Chossat and Guyard \cite{Chossat1994} get the isotropy classes of a direct sum of two irreducible $\SO(3)$ representations. To do so, they introduced a binary operations on conjugate $\SO(3)$-subgroups that allows one to compute the set of isotropy classes $\J(V)$ of a direct sum $V = V_1\oplus V_2$ of linear representations of a group $G$, if we know the isotropy classes for each individual irreducible representations. Such an operation is generalized to a binary operation on all conjugacy classes of subgroups of a given group $G$ and is defined as follows

\begin{defn}\label{def:clips_operation}
  For two subgroups $H_1$ and $H_2$ of a group $G$, we define the \emph{clips operation} of the conjugacy classes $[H_1]$ and $[H_2]$ as the subset of conjugacy classes
  \begin{equation*}
    [H_1]\circledcirc [H_2]:=\set{[H_1\cap gH_2g^{-1}],\quad g\in G}.
  \end{equation*}
  This definition immediately extends to two families (finite or infinite) $\mathcal{F}_1$ and $\mathcal{F}_2$ of conjugacy classes
  \begin{equation*}
    \mathcal{F}_1 \circledcirc \mathcal{F}_2=\underset{[H_i]\in \mathcal{F}_i}{\bigcup} [H_1]\circledcirc [H_2].
  \end{equation*}
\end{defn}

The following lemma states a central result that is useful to find the isotropy classes of a reducible representation once we know the isotropy classes of irreducible ones (see~\cite{Olive2019} for a proof).

\begin{lem}\label{lem:direct_sum}
  Let $V_1$ and $V_2$ be two linear representations of $G$. Then the set $\J(V_1\oplus V_2)$ of isotropy classes of the diagonal representation of $G$ on $V_1\oplus V_2$ is given by
  \begin{equation*}
    \J(V_1\oplus V_2)=\J(V_1)\circledcirc\J(V_2).
  \end{equation*}
\end{lem}

Using this result, one can find the isotropy classes of any representation $V$ provided we know
\begin{enumerate}
  \item a stable decomposition $V=W_1\oplus \dotsb \oplus W_r$;
  \item the isotropy classes of each representations $W_k$;
  \item the clips table of $[H_1]\circledcirc [H_2]$ for all subgroups of $G$.
\end{enumerate}

Lemma~\ref{lem:direct_sum} has already been applied to find the isotropy classes of some $\OO(3)$ reducible representations, which are essentially the standard $\OO(3)$ representations on \emph{odd order} tensor spaces on $\RR^3$ (see~\cite{Olive2014,Olive2019}). To extend these results to all reducible $\OO(3)$ representations, new clips table of closed $\OO(3)$-subgroups have to be calculated.

\section{Clips operation between closed $\OO(3)$-subgroups}
\label{sec:main_result}

The clips tables have already been established for two type I subgroups in \cite[table 1]{Chossat1994} and \cite[Table 1]{Olive2019} (see remark \ref{rem:Marc-vs-chossat}). As for two type III subgroups, the clips have already been calculated in ~\cite[Table 2]{Olive2019}). The clips operation between a type I and a type III subgroup is deduced from~\cite[Lemma 5.4]{Olive2019}. The clips between a type I and a type II subgroup or two type II subgroups are deduced from the clips between two type I subgroups, see remark \ref{rem:2typeII}.

As recalled in~\autoref{sec:O3-subgroups} (see also~\cite{Ihrig1984} for details), any closed $\OO(3)$-subgroup is either of type I, type II or type III. More specifically
\begin{itemize}
  \item Every type I closed $\OO(3)$-subgroup is conjugate to one of the following list
        \begin{equation}\label{eq:List_typeI}
          \SO(3),\quad \OO(2),\quad \SO(2), \quad \DD_n, \quad \ZZ_n, \quad \tetra, \quad \octa, \quad \ico, \text{ or} \quad \1.
        \end{equation}
  \item Every type II closed $\OO(3)$-subgroup is conjugate to $H\oplus \ZZ_2^c$ where $H$ is a type I subgroup and $\ZZ_2^c:=\set{-\id,\id}$.
  \item Every type III closed $\OO(3)$-subgroup is conjugate to one of the following list:
        \begin{equation}\label{eq:list_typeIII}
          \ZZ_{2n}^{-},\quad \Dnz,\quad \Dnd,\quad \octa^-,\quad \OO(2)^-.
        \end{equation}
\end{itemize}

Each closed subgroups from the lists~\eqref{eq:List_typeI} and~\eqref{eq:list_typeIII} is defined in \autoref{sec:O3-subgroups}.
\begin{rem}\label{rem:typeIII}
	Every type III subgroup of $\OO(3)$ is constructed from a pair of $\SO(3)$ subgroups $(\Gamma_+,\tilde{\Gamma})$ of index 2, where $\Gamma_+=\Gamma\cap\SO(3)$, such that 
	\begin{equation*}
	\Gamma=\Gamma_+\cup -\gamma \Gamma_+;\quad \gamma\in \tilde{\Gamma}\setminus \Gamma_+.
	\end{equation*}
	Note that $\tilde{\Gamma}=\Gamma_+\cup \gamma\Gamma_+$.
\end{rem}

Five subgroups of type III can be deduced
  \begin{align} \label{typeIII}
	& \ZZ_{2n}^-=\ZZ_n \cup (-\vR(\ee_3,\frac{\pi}{n})) \ZZ_n \quad \forall n \geq 1. \\
	& \Dnz=\ZZ_n \cup (-\vR(e_1,\pi))\ZZ_n  \quad \forall n\geq2.                     \\
	& \Dnd=\DD_n \cup (-\vR(\ee_3,\frac{\pi}{n})) \DD_n \quad \forall n \geq 1.       \\
	& \octa^-=\tetra \cup (-\vR(\ee_3,\frac{\pi}{2})) \tetra.                         \\
	& \OO(2)^-=\SO(2)\cup (-\vR(\ee_1,\pi)) \SO(2).
\end{align}

The notation $\vR(\nn,\theta)$, with $\theta\in [0;2\pi[$ and $\nn=(n_x,n_y,n_z)$, a unit vector, denotes the Rodrigues formula to represent a rotation by angle $\theta$ around $\nn$, which is given by
\begin{equation*}
	\vR(\nn,\theta)=\exp(\theta j(\nn))=\id+j(\nn)\sin(\theta)+j(\nn)^2(1-\cos(\theta))
\end{equation*}
where $j(\nn)$ denotes the antisymmetric matrix with entries
\begin{equation*}
	j(\nn)=\begin{pmatrix}
		0    & - n_z & n_y  \\
		n_z  & 0     & -n_x \\
		-n_y & n_x   & 0
	\end{pmatrix}.
\end{equation*}

As said before, the clips tables for two type I subgroups can be found in \cite[table 1]{Chossat1994} and \cite[Table 1]{Olive2019}. However, these two tables differ in some cases. In the following remark, we point out the differences between the two references and we recalculate the disputed clips.

\begin{rem}\label{rem:Marc-vs-chossat}
The two tables regrouping the clips between two $\SO(3)$-subgroups provided in \cite[table 1]{Chossat1994} and \cite[Table 1]{Olive2019} differ in the following cases:
\begin{itemize}
	\item Clips between $[\DD_n]$ and $[\OO(2)]$ (see equations \eqref{eq:Dn} and \eqref{eq:O(2)} for details on these subgroups): the table given by Chossat \cite{Chossat1994} is the correct one in this case
	\begin{equation*}
		[\DD_n]\circledcirc [\OO(2)]=\set{[\mathds{1}],[\DD_n],[\DD_2] (\text{if $n$ is even}), [\ZZ_2]}
	\end{equation*}
while Olive \cite{Olive2019} has $[\DD_2]$ if $n$ is odd. Indeed, the following intersection
	\begin{equation*}
	\DD_n\cap g\OO(2)g^{-1}=\set{\vR\left(\ee_3,\frac{2k\pi}{n}\right),\vR(\bb_i,\pi)}\cap \set{\vR(g\ee_3,\theta),\vR(g\bb,\pi)}
	\end{equation*}
gives $\DD_2$ only if $n$ is even (take for instance $g=\vR\left(\ee_2,\frac{\pi}{2}\right)$).

    \item Clips between $[\tetra]$ and $[\tetra]$: the table of Olive \cite[table 1]{Olive2019} is the correct one for this case 
    \begin{equation*}
    	[\tetra]\circledcirc[\tetra]=\set{[\mathds{1}],[\ZZ_2],[\DD_2],[\ZZ_3],[\tetra]}
    \end{equation*}
since Chossat \cite[table 1]{Chossat1994} omits the conjugacy class $[\DD_2]$. Indeed, $[\DD_2]\in [\tetra]\circledcirc [\tetra]$ since the intersection
\begin{equation*}
\tetra \cap g\tetra g^{-1}, \quad \text{ for } \tetra=\bigcup_{i=1}^3\ZZ_2^{\ee_i}\cup\bigcup_{j=1}^4 \ZZ_3^{\pmb{s}_{t_j}} \eqref{eq:Decomposition_tetra},
\end{equation*}
gives $\DD_2$ for $g=\vR\left(\ee_2,\frac{\pi}{2}\right)$ for example.

     \item Clips between $[\tetra]$ and $[\octa]$: we follow the table given by Olive \cite{Olive2019} for this case
     \begin{equation*}
     	[\tetra]\circledcirc[\octa]=\set{[\mathds{1}],[\ZZ_2],[\DD_2],[\ZZ_3],[\tetra]}
     \end{equation*}
 since Chossat \cite{Chossat1994} omits the conjugacy class $[\ZZ_3]$ . Indeed, $[\ZZ_3]\in [\tetra]\circledcirc [\octa]$ since the intersection
     \begin{equation*}
     \octa \cap g\tetra g^{-1}, \quad \text{ for  } \tetra=\bigcup_{i=1}^3\ZZ_2^{\ee_i}\cup\bigcup_{j=1}^4 \ZZ_3^{\pmb{s}_{t_j}} \eqref{eq:Decomposition_tetra},\text{ and }  \octa=\bigcup_{i=1}^3 \ZZ_4^{\ee_i}\cup\bigcup_{j=1}^4 \ZZ_3^{\pmb{s}_{t_j}}\cup\bigcup_{k=1}^6\ZZ_2^{\pmb{a}_{c_k}}\eqref{eq:Decomposition_Cube},
     \end{equation*}
     gives $\ZZ_3$ for $g=\vR\left(\ee_2-\ee_3,\arccos(-\sfrac{1}{3}) \right)$ for example.
     
     \item Clips between $[\octa]$ and $[\ico]$: we consider the table 1 of Olive \cite{Olive2019}
     
    \begin{equation*}
    	[\octa]\circledcirc[\ico]= \set{[\mathds{1}],[\ZZ_2],[\ZZ_3],[\DD_3],[\tetra]}
    \end{equation*} 
since $[\DD_2]\notin [\octa]\circledcirc [\ico]$; the intersection
     \begin{equation*}
     \octa \cap g\ico g^{-1}, \quad \text{ for } \ico=\bigcup_{i=1}^6 \ZZ_5^{\uu_i} \cup\bigcup_{j=1}^{10} \ZZ_3^{\vv_j}\cup\bigcup_{k=1}^{15}\ZZ_2^{\pmb{w}_k}\eqref{eq:Decomposition_Ico},
     \end{equation*}
     gives $\DD_2$ only if there exists three orthogonal $\pmb{w}_k,\ k=1,\dotsc, 15$, that turn to $\ee_1,\ee_2,\ee_3$ for some $g$. However, the only three orthogonal $\pmb{w}_k$ which match this condition are the triplet $(\pmb{w}_4,\pmb{w}_{10},\pmb{w}_{12})$, and in that case the intersection gives $\tetra$ (since $g$ is necessarily the identity rotation in that case). 
     
     \item Clips between $[\ico]$ and $[\ico]$: the table given by Olive \cite{Olive2019} is the correct one in this case
     \begin{equation*}
  [\ico]\circledcirc[\ico]=\set{[\mathds{1}],[\ZZ_2],[\ZZ_3],[\DD_3],[\ZZ_5],[\DD_5],[\ico]}
     \end{equation*}
     since Chossat has the conjugacy class $[\tetra]$ contrary to Olive which can't be realized by any rotation using the intersection $ \ico\cap g\ico g^{-1}$ (the only rotation that can realize $\tetra$ is the identity rotation which gives $\ico$ instead).

\end{itemize}
\end{rem}

\begin{rem}
  Note that
  \begin{enumerate}
    \item \label{rem:2typeII} If $H_1$ and $H_2$ are two closed subgroups of $\SO(3)$ then,
          \begin{equation*}
            [H_1]\circledcirc [H_2\oplus \ZZ_2^c]=[H_1]\circledcirc [H_2] \text{ and } [H_1\oplus \ZZ_2^c]\circledcirc [H_2\oplus \ZZ_2^c]=([H_1]\circledcirc [H_2])\oplus \ZZ_2^c.
          \end{equation*}
    \item For every closed subgroup $H$ of $\SO(3)$ we have
          \begin{equation*}
            [H]\circledcirc [\OO(3)]=\set{[H]} \text{ and } [\1]\circledcirc [H]=\set{[\1]}.
          \end{equation*}
  \end{enumerate}
\end{rem}

We give, in tables \ref{tab:Clips} and \ref{tab:Clips2}, clips operations between type II and type III subgroups, where we have used the notations
  \begin{equation*}
    d:=\mathrm{gcd}(m,n),\quad d_k:=\mathrm{gcd}(k,n),\quad d'_k:=\mathrm{gcd}(k,m),
  \end{equation*}
  \begin{equation*}\label{eq:Def_Gamma_n}
    \mathcal{Z}(n):=\begin{cases}
      [\ZZ_{2}]   & \text{ if } n \text{ even} \\
      [\ZZ_{2}^-] & \text{ else}               \\
    \end{cases} \text{ and  }   \mathsf{L}_{\octa}:=[\1],[\ZZ_2],[\DD_{d_2}],[\ZZ_2^-],[\DD_2^z],[\ZZ_{d_3}],[\DD_{d_3}],[\DD_{d_3}^z].
  \end{equation*}

\begin{rem}
  In tables we have used the conventions
  \begin{equation*}
    [\ZZ_1]:=[\1],\quad [\DD_1]=[\ZZ_2],\quad [\DD_1^z]=[\ZZ_2^-], \quad [\DD_2^z]=[\DD_2^d].
  \end{equation*}
\end{rem}

\begin{rem}
  Note that in \cite[figure 3]{Olive2019} it was stated that $\DD_2^z\subset\DD_{2n}^d$ only when $n$ is odd. However, this is true for all $n$.
\end{rem}

\begin{table}[H]
	\footnotesize
	\centering
	\begin{tabular}{|c||>{\centering}p{3.8cm}|>{\centering}p{4.2cm}|c|}
		\toprule
		$\circledcirc$
		&
		$[\ZZ_{2n}^-]$
		&
		$[\Dnz]$
		&
		$[\Dnd]$
		\\ \hline
		
		$[\ZZ_m \oplus \ZZ_2^c]$
		&
		\begin{tabular}{l}
			$ [\1],[\ZZ_{2d}^-] \text{ if } \frac{m}{d} \text{ even}$ \\  \\ \hline \\
			$ [\1],[\ZZ_{d}] \quad \text{\footnotesize{ else}}$
		\end{tabular}
		
		&
		\begin{tabular}{l}
			$ [\1],[\ZZ_d] \text{\footnotesize{ if }} m \text{\footnotesize{ odd}}$ \\ \\ \hline \\
			$ [\1],[\ZZ_{d}],[\ZZ_2^-] \text{\footnotesize{ else}}$
		\end{tabular}
		&
		\begin{tabular}{l}
			\\

			$[\1],[\ZZ_2],[\ZZ_2^-],[\ZZ_{2d}^-] $
			\footnotesize{ if $\frac{m}{d}$ even}
			\\ \\ \hline \\
			
			$[\1],[\ZZ_2],[\ZZ_2^-],
			[\ZZ_d] $
			
			\begin{tabular}{l}
				\footnotesize{if $m$ even} \\
				\footnotesize{and $\frac{m}{d}$ odd}
			\end{tabular}
			\\ \\ \hline \\
			$[\1],[\ZZ_{d}]$ \qquad \qquad \quad\footnotesize{  else}
			\rule[-0.5cm]{0cm}{0cm}
		\end{tabular}
		\\ \hline

		$[\DD_m \oplus \ZZ_2^c]$                 &
		
		\begin{tabular}{l}
			$[\1],[\ZZ_{2d}^-],\mathcal{Z}(n)$ \footnotesize{if $\frac{m}{d}$ even} \\ \\ \hline \\
			$[\1],[\ZZ_{d}],\mathcal{Z}(n)$ \footnotesize{else}
		\end{tabular}               &

		\begin{tabular}{l}
			\\
			\begin{tabular}{l}
				
				$[\1],[\ZZ_d],[\ZZ_{d_2}]$ \\
				$[\ZZ_2^-],[\DD_{d_2}^z],[\DD_d^z]$
			\end{tabular}
			\footnotesize{if $m$ even} \\ \\ \hline \\
			\begin{tabular}{l}
				$[\1],[\ZZ_d]$\\
				$[\ZZ_{d_2}],[\ZZ_2^-],[\DD_d^z]$
			\end{tabular}\footnotesize{ else}
			\rule[-0.6cm]{0cm}{0cm}
		\end{tabular}              &
		
		\begin{tabular}{l}
			\\
			\begin{tabular}{l}
				
				$[\1],[\ZZ_2],[\DD_{d_2}],[\ZZ_2^-]$ \\
				$[\ZZ_{2d}^-],[\DD_2^z],[\DD_{2d}^d] $
			\end{tabular}
			
			\quad \footnotesize{ if $\frac{m}{d}$ even}
			\\ \\ \hline \\
			\begin{tabular}{l}
				$[\1],[\ZZ_2],[\DD_2],[\ZZ_2^-]$ \\
				$[\DD_2^z],[\ZZ_d],[\DD_{d}],[\DD_{d}^z] $
			\end{tabular}
			\begin{tabular}{l}
				\footnotesize{if $m$ even} \\
				\footnotesize{and $\frac{m}{d}$ odd}
			\end{tabular}
			\\ \\ \hline \\
			\ \  $[\1],[\ZZ_2],[\ZZ_2^-],[\ZZ_d],[\DD_{d}],[\DD_{d}^z]$ \quad \footnotesize{ else}
			\rule[-0.5cm]{0cm}{0cm}
		\end{tabular}
		\\ \hline
		
		$[\octa \oplus \ZZ_2^c]$                 &
		
		\begin{tabular}{l}
			\\
			\begin{tabular}{l}
				$[\1],[\ZZ_2]$ \\ $[\ZZ_{d_3}],[\ZZ_4]$\end{tabular} \quad\footnotesize{if $4|n$} \\ \\ \hline \\
			\begin{tabular}{l} 	$[\1],[\ZZ_2]$ \\ $[\ZZ_{d_3}],[\ZZ_4^-]$  \end{tabular}
			\begin{tabular}{l} \footnotesize{if $n$ even} \\ \footnotesize{and $4\nmid n$} \end{tabular}                         \\ \\ \hline \\
			\ \ $[\1],[\ZZ_2^-],[\ZZ_{d_3}]$   \footnotesize{else}\end{tabular}              &
		
		\begin{tabular}{l}
			$[\1],[\ZZ_{d_2}],[\ZZ_{d_3}],[\ZZ_{d_4}]$ \\ $[\ZZ_2^-],[\DD_{d_2}^z],[\DD_{d_3}^z],[\DD_{d_4}^z]$
		\end{tabular}
		&

		\begin{tabular}{l}
			\\
			\begin{tabular}{l}
				$\mathsf{L}_{\octa},[\ZZ_4^-]$\\
				$[\ZZ_4],[\DD_4],[\DD_4^z]$
			\end{tabular}\quad \footnotesize{if $4|n$} \\ \\ \hline \\
			\begin{tabular}{l}
				$\mathsf{L}_{\octa},[\ZZ_4^-],[\DD_4^d]$
			\end{tabular}
			\begin{tabular}{l} \quad\footnotesize{if $n$ even} \\ \quad\footnotesize{and $4\nmid n$}
			\end{tabular}                        \\ \\ \hline \\
			\begin{tabular}{l}
				$\mathsf{L}_{\octa}$
			\end{tabular}  \quad \qquad\footnotesize{else}
			\rule[-0.5cm]{0cm}{0cm}
		\end{tabular}
		\\ \hline
		
		\rule[0.5cm]{0cm}{0cm}
		
		$[\tetra \oplus \ZZ_2^c]$
		&
		$[\1],[\ZZ_{d_3}],\mathcal{Z}(n)$
		&
		$[\1],[\ZZ_2^-],[\ZZ_{d_2}],[\ZZ_{d_3}],[\DD_{d_2}^z]$
		&
				\rule[0.5cm]{0cm}{0cm}
		$[\1],[\ZZ_2],[\ZZ_2^-],[\DD_{d_2}],[\DD_2^z],[\ZZ_{d_3}]$
		
		\rule[-0.5cm]{0cm}{0cm}
		\\ \hline
		
		\rule[0.5cm]{0cm}{0cm}
		$[\ico \oplus \ZZ_2^c]$                  &
		$[\1],\mathcal{Z}(n),[\ZZ_{d_3}],[\ZZ_{d_5}]$ &
		
		\begin{tabular}{l}
			\rule[0.5cm]{0cm}{0cm}
			$[\1],[\ZZ_{d_2}],[\ZZ_{d_3}]$\\$[\ZZ_{d_5}], [\ZZ_2^-],[\DD_{d_2}^z]$
			\rule[-0.5cm]{0cm}{0cm}
		\end{tabular}               &
		\begin{tabular}{l}
			$[\1],[\ZZ_2],[\ZZ_2^-],[\DD_{d_2}],[\DD_2^z]$,$[\ZZ_{d_3}],[\ZZ_{d_5}]$
		\end{tabular}
		\rule[-0.5cm]{0cm}{0cm}                                           \\
		\hline
		
		\rule[0.5cm]{0cm}{0cm}
		$[\SO(2)\oplus \ZZ_2^c]$                 &
		$[\1],[\ZZ_{2n}^-]$                           & $[\1],[\ZZ_2^-], [\ZZ_n]$ &
		$[\1],[\ZZ_2], [\ZZ_2^-],[\ZZ_{2n}^-]$
		
		\rule[-0.5cm]{0cm}{0cm}                                           \\
		\hline
		
		\rule[0.5cm]{0cm}{0cm}
		$[\OO(2)\oplus \ZZ_2^c]$                 &
		$[\1],\mathcal{Z}(n),[\ZZ_{2n}^-]$          &

		\rule[0.5cm]{0cm}{0cm}
		$ [\1],[\ZZ_2^-],[\DD_{d_2}^z],[\DD_n^z] $
		\rule[-0.5cm]{0cm}{0cm}                  &
		$[\1],[\ZZ_2],[\ZZ_2^-],[\DD_{d_2}],[\DD_2^z],[\Dnd]$
		\rule[-0.5cm]{0cm}{0cm}                                           \\
		\bottomrule
	\end{tabular}
	
	\caption{Clips between type II and III $\OO(3)$-subgroups (part 1)}
	\label{tab:Clips}
	
\end{table}

\begin{table}[H]
	\footnotesize
	\centering
	\begin{tabular}{|c||c|c|}
		\toprule
		$\circledcirc$
		&
		$[\octa^-]$
		&
		$[\OO(2)^-]$
		\\ \hline
		$[\ZZ_m \oplus \ZZ_2^c]$
		&
		\begin{tabular}{l}
			\\
			$[\1],[\ZZ_{d_3'}],[\ZZ_2^-],[\ZZ_4^-] $
			\footnotesize{ if $4|m$} \\ \\ \hline \\
			
			$[\1],[\ZZ_{d_2'}],[\ZZ_{d_3'}],[\ZZ_{d_2'}^-]$
			
			\footnotesize{ else}
			\rule[-0.5cm]{0cm}{0cm}
		\end{tabular}
		&
		$[\1],[\ZZ_m],[\ZZ_{d_2'}^-]$
		\\ \hline
		$[\DD_m \oplus \ZZ_2^c]$
		&
		
		\begin{tabular}{l}
			\\
			\begin{tabular}{l}
				$[\1],[\ZZ_2],[\ZZ_{d_3'}],[\ZZ_2^-]$ \\ $[\ZZ_4^-],[\DD_{d_3'}^z],[\DD_2^z],[\DD_4^d]$\end{tabular} \quad\footnotesize{if $4|m$} \\ \\ \hline \\
			\begin{tabular}{l} 	$[\1],[\ZZ_2],[\ZZ_{d_3'}],[\ZZ_2^-]$ \\ $[\DD_2],[\DD_{d_3'}^z],[\DD_2^z]$ \end{tabular}
			\begin{tabular}{l} \quad\footnotesize{if $m$ even} \\ \quad\footnotesize{and $4\nmid m$} \end{tabular}                         \\ \\ \hline \\
			\ \  $[\1],[\ZZ_2],[\ZZ_{d_3'}],[\ZZ_2^-],[\DD_{d_3'}^z]$ \quad \footnotesize{else}
			\rule[-0.5cm]{0cm}{0cm}\end{tabular}
		&
		\begin{tabular}{l}
			\\
			\begin{tabular}{l}
				
				$[\1],[\ZZ_2^-]$ \\
				$[\DD_2^z],[\DD_m^z]$
			\end{tabular}
			\footnotesize{if $m$ even} \\ \\ \hline \\
			\begin{tabular}{l}
				$[\1],[\ZZ_2]$ \\
				$[\ZZ_{2}^-],[\DD_m^z]$
			\end{tabular} \quad \footnotesize{else}
			\rule[-0.6cm]{0cm}{0cm}
		\end{tabular}
		\\ \hline
		
		\rule[0.5cm]{0cm}{0cm}
		$[\octa \oplus \ZZ_2^c]$
		&
		
		\rule[0.5cm]{0cm}{0cm}
		$[\1],[\ZZ_2],[\ZZ_{3}],[\ZZ_2^-],[\ZZ_4^-],[\DD_2^z],[\DD_3^z],[\DD_4^d],[\octa^-]$
		\rule[-0.5cm]{0cm}{0cm}

		&
		
		\rule[0.5cm]{0cm}{0cm}
		$[\1],[\ZZ_2],[\ZZ_2^-],[\DD_2^z],[\DD_3^z],[\DD_4^z]$
		\rule[-0.5cm]{0cm}{0cm}
		
		\\ \hline
		
		\rule[0.5cm]{0cm}{0cm}
		
		$[\tetra \oplus \ZZ_2^c]$
		&
		$[\1],[\ZZ_2],[\ZZ_{3}],[\ZZ_2^-],[\DD_2],[\DD_2^z],[\tetra]$
		&
		$[\1], [\ZZ_2], [\ZZ_3], [\ZZ_2^-],[\DD_2^z]$
		\rule[-0.5cm]{0cm}{0cm}
		
		\\ \hline
		
		\rule[0.5cm]{0cm}{0cm}
		$[\ico \oplus \ZZ_2^c]$
		&
		$[\1],[\ZZ_2],[\ZZ_2^-],[\DD_2],[\DD_2^z],[\ZZ_{3}],[\tetra]$
		&
		
		\rule[0.5cm]{0cm}{0cm}
		$[\1],[\ZZ_2],[\ZZ_2^-],[\DD_2^z]$
		\rule[-0.5cm]{0cm}{0cm}
		
		\\
		\hline
		
		\rule[0.5cm]{0cm}{0cm}
		$[\SO(2)\oplus \ZZ_2^c]$         &
		$[\1],[\ZZ_3],[\ZZ_2^-],[\ZZ_4^-]$    &
		$ [\1],[\ZZ_2^-],[\SO(2)]$
		
		\rule[-0.5cm]{0cm}{0cm}            \\
		\hline
		
		\rule[0.5cm]{0cm}{0cm}
		$[\OO(2)\oplus \ZZ_2^c]$         &
		$[\1],[\ZZ_2^-],[\DD_3^z], [\DD_4^d]$ &
		$ [\DD_2^z],[\OO(2)^-]$
		\rule[-0.5cm]{0cm}{0cm}            \\
		\bottomrule
	\end{tabular}
	
	\caption{Clips between type II and III $\OO(3)$-subgroups (part 2)}
	\label{tab:Clips2}
	
\end{table}

\section{Proofs}
\label{sec:proof_clips}

In this section, we provide the details of the computation of clips between type II and type III $\OO(3)$ subgroups. We recall from \autoref{sec:O3-subgroups} that every type III subgroup of $\OO(3)$ is constructed from a pair of $\SO(3)$ subgroups $(\Gamma_+,\tilde{\Gamma})$ of index 2, where $\Gamma_+=\Gamma\cap\SO(3)$, such that 
	\begin{equation*}
	\Gamma=\Gamma_+\cup -\gamma \Gamma_+;\quad \gamma\in \tilde{\Gamma}\setminus \Gamma_+.
	\end{equation*}
	Note that $\tilde{\Gamma}=\Gamma_+\cup \gamma\Gamma_+$.

Let $\Gamma=\Gamma_+\cup (-\gamma \Gamma_+)$ be a subgroup of type III and $H$ a subgroup of $\SO(3)$.
By definition, the clips operation between a type II and a type III subgroups $[\Gamma] \circledcirc [H\oplus \ZZ_2^c]$ is given by the intersection $\Gamma\cap (g H g^{-1}\oplus \ZZ_2^c)$ wich can be reduced to 
\begin{equation}\label{eq:Inter_Type_I_et_III}
	\Gamma\cap (g H g^{-1}\oplus \ZZ_2^c)=(
	\Gamma_+\cap (gHg^{-1}))\cup (-(\gamma \Gamma_+\cap (gHg^{-1}))),\quad -\gamma\in \Gamma \setminus \Gamma_+
\end{equation}
Indeed, as we have $H\oplus \ZZ_2^c=H\cup (-H)$ we deduce that
		\begin{align*}
		\Gamma\cap (H\oplus \ZZ_2^c) & =(\Gamma_+\cup(-\gamma \Gamma_+))\cap(H\cup (-H))                                                        \\
		& =(\Gamma_+\cap H)\cup (\Gamma_+\cap(-H))\cup (-(\gamma \Gamma_+)\cap H)\cup (-(\gamma \Gamma_+)\cap(-H)) \\
		& =(\Gamma_+\cap H)\cup (-((\gamma \Gamma_+)\cap H)).
		\end{align*}

In the following, we prove a theorem which describes the possible conjugacy classes belonging to the clips between a type II and type III subgroups of $\OO(3)$. The computation of such classes can be complicated since it involves intersection of groups of type III which can be tricky sometimes and we risk to not cover all the groups resulting from this intersection. For this reason, we use a characterization of these conjugacy classes making use of the clips between type I subgroups that have already been calculated before. But, before, a preparatory lemma is needed.

\begin{lem}\label{lem:index2}
	Let $(K_+,\tilde{K})$ be a pair of subgroups of a group $G$ of index 2 and $H$ be a subgroup of $G$. Then, either
	\begin{equation*}
		K_+\cap H= \tilde{K}\cap H, \text{or the pair } (K_+\cap H,\tilde{K}\cap H) \text{ is of index 2}.
	\end{equation*} 
\end{lem}

\begin{proof}
		Consider the following exact sequence
		\begin{equation*}
			\1\to K_+\overset{i}{\to} \tilde{K}\overset{p}{\to} \tilde{K}/K_+\simeq \ZZ_2\to \1
		\end{equation*}
	Then $p(\tilde{K}\cap H)$ is a subgroup of $\ZZ_2$. Therefore, either  $p(\tilde{K}\cap H)=\ZZ_2$ and in this case $(K_+\cap H, \tilde{K}\cap H)$ is a pair of index 2 (considering the exact sequence $\1\to K_+\cap H\to \tilde{K}\cap H\to \ZZ_2$) or  $p(\tilde{K}\cap H)=\1$ and then $\tilde{K}\cap H=K_+\cap H$.
\end{proof}

\begin{thm}\label{thm:main}
	Let $\Gamma$ be a type III subgroup of $\OO(3)$, built from the pair of $\SO(3)$ subgroups $(\Gamma_+,\tilde{\Gamma})$ of index 2 and let $H$ be a subgroup of $\SO(3)$. Let $L$ be a subgroup of $\OO(3)$ such that $[L]\in [\Gamma]\circledcirc[H\oplus\ZZ_2^c]$ then, either
	\begin{enumerate}
		\item $L \text{ is of type I and } [L]\in ([\Gamma_+]\circledcirc [H])\cap ([\tilde{\Gamma}]\circledcirc[H])$, 
		\item $\text{ Or } L \text{ is of type III built from a pair } (L_+,\tilde{L})  \text{ of index 2 such that } [L_+]\in [\Gamma_+]\circledcirc[H] \text{ and }[\tilde{L}]\in [\tilde{\Gamma}]\circledcirc[H]$.
	\end{enumerate}	
\end{thm}

\begin{proof}
	First remark that $-\id$ doesn't belong to the intersection between a type II and a type III subgroups and hence if $[L]\in [\Gamma]\circledcirc[H\oplus\ZZ_2^c]$ then $L$ is either of type I or type III. 
	Let $[L]\in [\Gamma]\circledcirc[H\oplus\ZZ_2^c]$ then $[L]$ is given by the union \eqref{eq:Inter_Type_I_et_III}
	\begin{equation*} 
	L=\Gamma_+\cap gHg^{-1}\cup (-(\gamma\Gamma_+\cap gHg^{-1}) \quad g\in \SO(3),\ \gamma\in \tilde{\Gamma}\setminus\Gamma_+.
	\end{equation*} 
	If $\gamma\Gamma_+\cap gH g^{-1}=\emptyset$ then $L$ is of type I and $L=\Gamma_+\cap gHg^{-1}=\tilde{\Gamma}\cap gHg^{-1}$. Hence, $[L]\in ([\Gamma_+]\circledcirc [H])\cap ([\tilde{\Gamma}]\circledcirc[H])$. 
	
	If  $\gamma\Gamma_+\cap gH g^{-1}\neq\emptyset$ then $L$ is of type III and $\Gamma_+\cap gHg^{-1}\neq \tilde{\Gamma}\cap gHg^{-1}$. Since $(\Gamma_+,\tilde{\Gamma})$ is a pair of index 2 then, by lemma \ref{lem:index2}, the pair $(\Gamma_+\cap gHg^{-1},\tilde{\Gamma}\cap gHg^{-1})$ is of index 2. Hence, there exists $\sigma\in\tilde{\Gamma}\cap gHg^{-1}\setminus (\Gamma_+\cap gHg^{-1})$ such that
	\begin{align*}
	L=\Gamma_+\cap gHg^{-1}\cup (-\sigma(\Gamma_+\cap gHg^{-1}))	\end{align*}
	Hence, the result.
\end{proof}

In the next subsections, we compute the clips operation between type II and type III subgroups of $\OO(3)$ using theorem \ref{thm:main} which involves  clips between type I $\OO(3)$-subgroups that have already been calculated in previous works (\cite{Chossat1994,Olive2019}, see also remark \ref{rem:Marc-vs-chossat}). Indeed, the classes for the clips operation between a subgroup $H\oplus \ZZ_2^c$ of type II and a subgroup $\Gamma$ of type III are deduced from the knowledge of the tables of the clips $[\Gamma_+]\circledcirc [H]$ and $[\tilde{\Gamma}]\circledcirc[H]$. By eliminating the classes that can't be realised by any $g\in \SO(3)$, we deduce the clips operation between $[\Gamma]$ and $[H\oplus \ZZ_2^c]$. 



\subsection{Clips with $\ZZ_{2n}^-$}

First, let us recall that $\ZZ_{2n}^-$ is built from the couple $(\ZZ_{n},\ZZ_{2n})$, where $\ZZ_n$ is given by \eqref{eq:Zn}, as in the equation \eqref{typeIII} of appendix \ref{sec:O3-subgroups}
\begin{equation*}
  \ZZ_{2n}^-=\ZZ_{n}\cup (-\gamma\ZZ_n),\quad \gamma=\vR\left(\ee_3,\frac{\pi}{n}\right)
\end{equation*}
where we note $\ZZ_1:=\set{Id}$ when $n=1$.
\begin{lem}\label{lem:clips_Zm}
  Let $n \geq 1$ and $m\geq 2$ be two integers and $d=\gcd(m,n)$. Then
  \begin{align*}
    [\ZZ_{2n}^{-}] \circledcirc [\ZZ_m \oplus \ZZ_2^c]=
    \begin{cases}
      \set{[\1],[\ZZ_{2d}^-]} & \text{ if  $\frac{m}{d}$ even} \\
      \set{[\1],[\ZZ_d]}      & \text{ else}
    \end{cases}
    .
  \end{align*}
\end{lem}

\begin{proof}
  We deduce from \eqref{eq:Inter_Type_I_et_III} that
  \begin{equation*}
    \ZZ_{2n}^- \cap (g\ZZ_m g^{-1}\oplus \ZZ_2^c)=(\ZZ_n\cap g \ZZ_m g^{-1})\cup (-(\gamma\ZZ_n\cap g \ZZ_m g^{-1})),\quad \gamma=\vR\left(\ee_3,\frac{\pi}{n}\right).
  \end{equation*}

  In the case when $g\ee_3$ and $\ee_3$ are not colinear, such group reduces to $\1$, so we suppose now that $g\ee_3=\pm \ee_3$. We thus have to consider
  \begin{equation*}
    (\ZZ_n\cap \ZZ_m )\cup (-(\gamma\ZZ_n\cap \ZZ_m))
  \end{equation*}
  where $\ZZ_n\cap \ZZ_m=\ZZ_d$ and $\gamma\ZZ_n\cap \ZZ_m$ is obtained by solving the equations of unknown $k_1,k_2\in \ZZ$
  \begin{equation*}
    \frac{2k_1+1}{n}=\frac{2k_2}{m} \iff (2k_1+1)d m_1=2k_2d n_1,\quad \mathrm{gcd}(m_1,n_1)=1.
  \end{equation*}
  We get solutions only if $m_1=\frac{m}{d}$ is even. By replacing $m_1$ by $2p$ we get that $p$ devides $k_2$ hence $k_2=p k'$ with $k'$ odd.
  On one hand, $\frac{2k_2}{m}=\frac{2p k'}{2pd}$ and on the other hand, $\frac{2k_1+1}{n}=\frac{k' n_1 }{dn_1}$. We deduce that

  \begin{equation*}
    \gamma\ZZ_n\cap \ZZ_m=\set{\vR\left(\ee_3,\frac{(2k+1)\pi}{d}\right)}
  \end{equation*}
\end{proof}

In the following, recall that
\begin{equation}\label{eq:Def_Zn}
  \mathcal{Z}(n):=\ZZ_{2n}^{-}\circledcirc (\ZZ_2\oplus \ZZ_2^c)= \begin{cases}
    [\ZZ_{2}]   & \text{ if } n \text{ even} \\
    [\ZZ_{2}^-] & \text{ else}               \\
  \end{cases}.
\end{equation}

\begin{lem}\label{lem:Z2nmoinsclipsDm}
  Let $n\geq 1$ and $m\geq 2$ be two integers and $d=\gcd(m,n)$. Then
  \begin{equation*}
    [\ZZ_{2n}^{-}] \circledcirc [\DD_m \oplus \ZZ_2^c]=
    \begin{cases}
      \set{[\1],[\ZZ_{2d}^-],\mathcal{Z}(n)} & \text{ if $\frac{m}{d}$ even} \\
      \set{[\1],[\ZZ_d],\mathcal{Z}(n)}      & \text{ else}
    \end{cases}		.
  \end{equation*}
\end{lem}

\begin{proof}
  Recall from \eqref{eq:Dn} that
  \begin{equation*}
    \DD_m=\ZZ_m \cup_{i=1}^{m} \ZZ_2^{\bb_i},\quad \ZZ_{2}^{\bb_i}:=\set{e,\vR(\bb_i,\pi)}
  \end{equation*}
  where $\bb_i$ are called the secondary axes of the subgroup $\DD_m$, with
  \begin{equation*}
    \bb_1=\ee_1,\quad \bb_k=\vR\left(\ee_3,\frac{\pi}{m}\right)\bb_{k-1},\quad k=2,\dotsc,m.
  \end{equation*}
  We have
  \begin{align*}
    \ZZ_{2n}^- \cap (g\DD_m g^{-1}\oplus \ZZ_2^c)=\left(\ZZ_n\cap g\DD_m g^{-1} \right) \bigcup \left(-(\gamma\ZZ_n \cap(g\DD_m g^{-1})\right).
   \end{align*}
  The non tivial cases take place for $g\ee_3=\pm\ee_3$ and the union is deduced from lemma \ref{lem:clips_Zm}, and for $g\bb_i=\pm\ee_3$ in which case we get $\mathcal{Z}(n)$, so we can conclude.
\end{proof}

\begin{lem}\label{lem:Z2nmClipsExceptionels}
Let $d_3=\gcd(3,n)$ and $d_5=\gcd(5,n)$.
	  We have
	  \begin{align*}
		   [\ZZ_{2n}^-]\circledcirc [\tetra\oplus\ZZ_2^c] & =
		   \set{[\1],[\ZZ_{d_3}],\mathcal{Z}(n)},\quad [\ZZ_{2n}^-]\circledcirc[\ico\oplus \ZZ_2^c]=\set{[\1],\mathcal{Z}(n),[\ZZ_{d_3}],[\ZZ_{d_5}]}. \\
		    [\ZZ_{2n}^-]\circledcirc[\octa\oplus \ZZ_2^c]  & =
		    \begin{cases}
		      \set{[\1],[\ZZ_2],[\ZZ_{d_3}],[\ZZ_4]}   & \text{ if $n$ is even and $4|n$}      \\
		     \set{[\1],[\ZZ_2],[\ZZ_{d_3}],[\ZZ_4^-]} & \text{ if $n$ is even but $4\nmid n$} \\
		     \set{[\1],[\ZZ_2^-],[\ZZ_{d_3}]}         & \text{ if $n$ is odd}
		   \end{cases}.
		 \end{align*}
		\end{lem}
		
\begin{proof}
	By \eqref{eq:Inter_Type_I_et_III}, we have to consider
	\begin{equation*}
(\ZZ_n\cap (gH g^{-1}))\cup (-(\gamma \ZZ_n\cap (gH g^{-1}))) \text{ for } H=\tetra,\octa,\ico.
 \end{equation*}
 Let us consider decomposition \eqref{eq:Decomposition_tetra} of the group $\tetra$. Then the only non trivial cases are obtained for $g\ee_i=\pm\ee_3$ in which case we get $\mathcal{Z}(n)$ and for $g\bs_{t_j}=\pm \ee_3$ and we get $\ZZ_{d_3}$.
 
 For the group $\octa$ given in \eqref{eq:Decomposition_Cube}, the only non trivial cases are for $g\ba_{c_k}=\pm \ee_3$ and $g\bs_{t_j}=\pm\ee_3$ which have been calculated in the case of $\tetra$ and for $g\ee_i=\pm \ee_3$ in which case we get $\ZZ_4$ if $4\mid n$, $\ZZ_4^-$ if $n$ is even but $4 \nmid n$ and $\ZZ_2^-$ otherwise.

 For the group $\ico$, we consider the decomposition \eqref{eq:Decomposition_Ico} and we deduce the result following the same reasonning.
\end{proof}

Finally, for the subgroups $\SO(2)$ and $\OO(2)$ given in \eqref{eq:SO(2)} and \eqref{eq:O(2)} and as a direct consequence of \eqref{eq:Inter_Type_I_et_III},

\begin{lem}
  For any integer $n\geq 1$ we have:
  \begin{equation*}
    [\ZZ_{2n}^-]\circledcirc [\SO(2)\oplus \ZZ_2^c]=\set{[\1],[\ZZ_{2n}^-]},\quad 	[\ZZ_{2n}^-]\circledcirc [\OO(2)\oplus \ZZ_2^c]=\set{[\1],\mathcal{Z}(n),[\ZZ_{2n}^-]}
  \end{equation*}
  with $\mathcal{Z}(n)$ given by~\eqref{eq:Def_Zn}.
\end{lem}


\subsection{Clips with $\Dnz$}
\label{sec:withDnz}
As explained in \autoref{sec:O3-subgroups}, the subgroup $\Dnz$ is obtained for $-\gamma=-\vR(\be_1,\pi)$, so that
\begin{equation*}
  \Dnz=\ZZ_n \cup -\gamma \ZZ_n=\ZZ_n\cup \set{-\vR(\bb_1,\pi),\dotsc,\ -\vR(\bb_n,\pi)}, \quad \bb_1=\ee_1,\quad \bb_k=\vR\left(\ee_3,\frac{\pi}{n}\right)\bb_{k-1}
\end{equation*}
where $\langle \bb_i \rangle$ are called secondary axes of the dihedral subgroup $\DD_n$ (see~\cite[Appendix A]{Olive2019}).

\begin{lem}\label{lem:DnvClipsZm}
  Let $m,\ n \geq 2$ be two integers and $d=\gcd(m,n)$. Then
  \begin{align*}
    [\Dnz] \circledcirc [\ZZ_m \oplus \ZZ_2^c]=
    \begin{cases}
      \set{[\1],[\ZZ_d]}             & \text{ if $m$ is odd}  \\
      \set{[\1],[\ZZ_{2}^-],[\ZZ_d]} & \text{ if $m$ is even}
    \end{cases}.
  \end{align*}
\end{lem}
\begin{proof}
  From \eqref{eq:Inter_Type_I_et_III} we have to consider
  \begin{align*}
    \Dnz \cap (g\ZZ_m g^{-1} \oplus \ZZ_2^c)=(\ZZ_n\cap g \ZZ_m g^{-1})\cup (-(\gamma\ZZ_n\cap g \ZZ_m g^{-1})),\quad \gamma=\vR(\be_1,\pi),
  \end{align*}
  where $\gamma \ZZ_n=\set{\vR(\bb_1,\pi),\dotsc,\ \vR(\bb_n,\pi)}$.
  Here, the only non--trivial cases are obtained when $g\ZZ_mg^{-1}=\ZZ_n$ (see lemma \ref{lem:clips_Zm}) or $\ZZ_2^{\bb_i}$ , which leads directly to the result.
\end{proof}

\begin{lem}\label{lem:DnvClipsDm}
  Let $m,\ n \geq 2$ be two integers and $d=\gcd(m,n)$ and $d_2=\gcd(2,n)$. Then
  \begin{align*}
    [\DD_n^z]\circledcirc [\DD_m\oplus \ZZ_2^c]=
    \begin{cases}
      \set{[\1],[\ZZ_{d_2}],[\ZZ_d],[\ZZ_2^-],[\DD_{d_2}^z],[\DD_d^z]} \quad & \text{if $m$ is even} \\
      \set{[\1],[\ZZ_{d_2}],[\ZZ_d],[\ZZ_2^-],[\DD_d^z]} \quad   & \text{if $m$ is odd}
    \end{cases} .
  \end{align*}
\end{lem}
\begin{proof}
We apply theorem \ref{thm:main} with $\Gamma_+=\ZZ_n$ and $\tilde{\Gamma}=\DD_n$. We have from \cite[table 1]{Olive2019}
\begin{equation*}
[\ZZ_n]\circledcirc [\DD_m]=\set{[\1],[\ZZ_{d_2}],[\ZZ_d]} \text{ and }
[\DD_n]\circledcirc [\DD_m]=\set{[\1],[\ZZ_2],[\DD_2] (\text{ if $m$ and $n$ even}),[\ZZ_d],[\DD_d]}
\end{equation*}
Hence, the classes of $[\DD_n^z]\circledcirc [\DD_m\oplus \ZZ_2^c]$, corresponding to type I subgroups, are among the following list 
	 \begin{equation*}
([\Gamma_+]\circledcirc [H])\cap ([\tilde{\Gamma}]\circledcirc[H])=\set{[\1],[\ZZ_{d_2}],[\ZZ_d]}
	 \end{equation*}
	 and the classes corresponding to type III subgroups are among
	 \begin{equation*}
	 \set{[\ZZ_2^-],[\DD_{d_2}^z] (\text{ if $m$ is even}),[\DD_d^z]}.
	 \end{equation*}
	 We can check that all the possibilities can occur by using \eqref{eq:Inter_Type_I_et_III}:
	 \begin{equation*}
	 	 (\ZZ_n\cap g \DD_m g^{-1})\cup (-(\gamma\ZZ_n\cap g \DD_m g^{-1}))
	 \end{equation*}
 where $ \gamma\ZZ_n=\set{\vR(\bb_i,\pi),i=1,\dotsc,n}$ and $\DD_m=\ZZ_m \cup_{j=1}^{m} \ZZ_2^{\bb_j}$,
 we get
	 \begin{itemize}
	 	\item $[\ZZ_{d_2}]$ for a rotation $g$ such that $g\bb_j= \ee_3$ for $j=1,\dotsc,m$ (take for instance $g=\vR\left(\ee_3,\frac{\pi}{3}\right)\circ\vR\left(\ee_2,\frac{\pi}{2}\right)$ that turns only $\ee_1$ to $\ee_3$ and nothing else);
	 	\item $[\ZZ_d]$ for a rotation $g$ such that $g\ee_3= \ee_3$ (for instance $g=\vR\left(\ee_3,\frac{\pi}{3}\right)$);
	 	\item $[\ZZ_2^-]$ for a rotation $g$ such that $g\bb_j=\bb_i$ for $ i=1,\dotsc n$ and $j=1,\dotsc,m$ (for instance $g=\vR\left(\ee_1,\frac{\pi}{3}\right)$);
	 	\item $[\DD_d^z]$ for the identity rotation for instance;
	 	\item $[\DD_{d_2}^z]$ if $m$ is even and in this case we can take $g=\vR\left(\ee_2,\frac{\pi}{2}\right)$ for instance so we have $g\ee_3=\ee_1=\bb_1$ and $g\bb_2=\bb_2=\ee_2$.
	 \end{itemize}
\end{proof}

\begin{lem}
  Let $n\geq 2$ be an integer, $d_2=\gcd(2,n)$ and $d_3=\gcd(3,n)$. We have
  \begin{equation*}
  [\Dnz] \circledcirc [\tetra \oplus \ZZ_2^c]  =\set{[\1],[\ZZ_{2}^{-}],[\ZZ_{d_2}],[\ZZ_{d_3}],[\DD_{d_2}^{z}]}.
  \end{equation*}
\end{lem}

\begin{proof}
	We deduce from \cite[table 1]{Olive2019} and theorem \ref{thm:main} that the classes in $[\DD_n^z]\circledcirc [\tetra\oplus \ZZ_2^c]$, corresponding to type I subgroups, are among the following list
\begin{equation*}
([\ZZ_n]\circledcirc [\tetra])\cap ([\DD_n]\circledcirc[\tetra])=\set{[\1],[\ZZ_{d_2}],[\ZZ_{d_3}]}
\end{equation*}
and the classes corresponding to type III subgroups are among
\begin{equation*}
\set{[\ZZ_2^-],[\DD_{d_2}^z]}.
\end{equation*}
We can check that all the possibilities can occur by using \eqref{eq:Inter_Type_I_et_III}:
\begin{equation*}
(\ZZ_n\cap g \tetra g^{-1})\cup (-(\gamma\ZZ_n\cap g \tetra g^{-1}))
\end{equation*}
where $ \gamma\ZZ_n=\set{\vR(\bb_i,\pi),i=1,\dotsc,n}$ and $\tetra=\bigcup_{i=1}^3\ZZ_2^{\ee_i}\bigcup_{j=1}^4 \ZZ_3^{\pmb{s}_{t_j}}$ \eqref{eq:Decomposition_tetra},
we get
\begin{itemize}

	\item $[\ZZ_2^-]$ for a rotation $g$ such that $g\ee_1=\bb_i$ for $ i=1,\dotsc n$ (take for instance $g=\vR\left(\ee_1,\frac{\pi}{3}\right)$);
	\item $[\ZZ_{d_2}]$ for a rotation $g$ such that $g\ee_i= \ee_3$ for $i=1,\dotsc,3$ (take for instance $g=\vR\left(\ee_3,\frac{\pi}{3}\right)$);
	\item $[\ZZ_{d_3}]$ for a rotation $g$ such that $g\pmb{s}_{t_j}= \ee_3$;
	\item $[\DD_{d_2}^z]$ for a rotation $g$ such that $g\ee_i=\ee_3$ for $i=1,2,3$ and the two remaining $\ee_i$ turn to two orthogonal $\bb_{i}$ (exists for $n$ even).
\end{itemize}
\end{proof}

\begin{lem}
Let $n\geq 2$ be an integer, $d_2=\gcd(2,n)$, $d_3=\gcd(3,n)$ and $d_4=\gcd(4,n)$. We have
		\begin{equation*}
		[\Dnz] \circledcirc [\octa \oplus \ZZ_2^c]   =\set{[\1],[\ZZ_{d_2}],[\ZZ_{d_3}],[\ZZ_{d_4}],[\ZZ_{2}^{-}],[\DD_{d_2}^{z}],[\DD_{d_3}^z],[\DD_{d_4}^z]  }.
		\end{equation*}
	
	\end{lem}

	\begin{proof}
		We deduce from \cite[table 1]{Olive2019} and theorem \ref{thm:main} that the classes in $[\DD_n^z]\circledcirc [\octa\oplus \ZZ_2^c]$, corresponding to type I subgroups, are among the following list
		\begin{equation*}
		([\ZZ_n]\circledcirc [\octa])\cap ([\DD_n]\circledcirc[\octa])=\set{[\1],[\ZZ_{d_2}],[\ZZ_{d_3}],[\ZZ_{d_4}]}
		\end{equation*}
		and the classes corresponding to type III subgroups are among
		\begin{equation*}
		\set{[\ZZ_2^-],[\DD{d_2}^z] ,[\DD_{d_3}^z],[\DD_{d_4}^z]}.
		\end{equation*}
		We can check that all the possibilities can occur by using \eqref{eq:Inter_Type_I_et_III}:
		\begin{equation*}
		(\ZZ_n\cap g \octa g^{-1})\cup (-(\gamma\ZZ_n\cap g \octa g^{-1}))
		\end{equation*}
		where $\octa=\bigcup_{i=1}^3 \ZZ_4^{\ee_i}\bigcup_{j=1}^4 \ZZ_3^{\pmb{s}_{t_j}}\bigcup_{k=1}^6\ZZ_2^{\pmb{a}_{c_k}}$ \eqref{eq:Decomposition_Cube},
		we get
		\begin{itemize}
			
			\item $[\ZZ_2^-]$ for a rotation $g$ such that $g\pmb{a}_{c_k}=\bb_i$ for $ i=1,\dotsc n$ and $k=1,\dotsc,6$ (for instance take $g=\vR\left(\ee_1,\frac{\pi}{3}\right)\circ \vR\left(\ee_3,-\frac{\pi}{4}\right)$ that turns only $\pmb{a}_{c_1}$ to $\ee_1$);
			\item $[\ZZ_{d_2}]$ for a rotation $g$ such that $g\pmb{a}_{c_k}= \ee_3$ for $k=1,\dotsc,6$;
			\item $[\ZZ_{d_3}]$ for a rotation $g$ such that $g\pmb{s}_{t_j}= \ee_3$;
			\item $[\ZZ_{d_4}]$ for a rotation $g$ such that $g\ee_i= \ee_3$ for $i=1,\dotsc,3$ (take for instance $g=\vR\left(\ee_3,\frac{\pi}{3}\right)\circ\vR\left(\ee_1,\frac{\pi}{2}\right)$);
		    \item $[\DD_{d_2}^z]$ for a rotation $g$ such that $g\pmb{a}_{c_k}= \ee_3$ and two other edge axes $\pmb{a}_{c_k}$  turn to two orthogonal axes $\pm \bb_i$ for $i=1,\dotsc,n$ (exists for $n$ even) ;	   
		    \item $[\DD_{d_3}^z]$ for a rotation $g$ such that $g\pmb{s}_{t_j}= \ee_3$ and three other edge axes $\pmb{a}_{c_k}$ (for instance $\pmb{a}_{c_1},\pmb{a}_{c_4}$ and $\pmb{a}_{c_5}$ are three coplanar edge axes seperated by an angle of $\frac{\pi}{3}$) turn to three $\pm \bb_i$;
		    \item $[\DD_{d_4}^z]$ for $g=\vR\left(\ee_3,\frac{\pi}{4}\right)$ for example.
		\end{itemize}
	\end{proof}

\begin{lem}\label{lem:Dnz-with_ico}
Let $n\geq 2$ be an integer, $d_2=\gcd(2,n)$, $d_3=\gcd(3,n)$ and $d_5=\gcd(5,n)$. We have
		\begin{equation*}
	[\Dnz] \circledcirc [\ico \oplus \ZZ_2^c]   =\set{[\1],[\ZZ_{d_2}],[\ZZ_{d_3}],[\ZZ_{d_5}],[\ZZ_{2}^{-}],[\DD_{d_2}^z]}.
		\end{equation*}
		
	\end{lem}

	\begin{proof}
		We deduce from \cite[table 1]{Olive2019} and theorem \ref{thm:main} that the classes in $[\DD_n^z]\circledcirc [\ico\oplus \ZZ_2^c]$, corresponding to type I subgroups, are among the following list
		\begin{equation*}
		([\ZZ_n]\circledcirc [\ico])\cap ([\DD_n]\circledcirc[\ico])=\set{[\1],[\ZZ_{d_2}],[\ZZ_{d_3}],[\ZZ_{d_5}]}
		\end{equation*}
		and the classes corresponding to type III subgroups are among
		\begin{equation*}
		\set{[\ZZ_2^-],[\DD_{d_2}^z] ,[\DD_{d_3}^z],[\DD_{d_5}^z]}.
		\end{equation*}
		We can check that all the possibilities can occur, except $[\DD_{d_3}^z]$ and $[\DD_{d_5}^z]$, by using \eqref{eq:Inter_Type_I_et_III}:
		\begin{equation*}
		(\ZZ_n\cap g \ico g^{-1})\cup (-(\gamma\ZZ_n\cap g \ico g^{-1}))
		\end{equation*}
		where $ \ico=\bigcup_{i=1}^6 \ZZ_5^{\uu_i}\bigcup_{j=1}^{10} \ZZ_3^{\vv_j}\bigcup_{k=1}^{15}\ZZ_2^{\pmb{w}_k}$ \eqref{eq:Decomposition_Ico},
		we get
		\begin{itemize}
			
			\item $[\ZZ_2^-]$ for a rotation $g$ such that $g\pmb{w}_{k}=\bb_i$ for $ i=1,\dotsc n$ and $k=1,\dotsc,15$;
			\item $[\ZZ_{d_2}]$ for a rotation $g$ such that $g\pmb{w}_{k}= \ee_3$ for $k=1,\dotsc,15$;
			\item $[\ZZ_{d_3}]$ for a rotation $g$ such that $g\vv_{j}= \ee_3$ for $j=1,\dotsc,10$;
			\item $[\ZZ_{d_5}]$ for a rotation $g$ such that $g\uu_i= \ee_3$ for $i=1,\dotsc,6$;
			\item $[\DD_{d_2}^z]$ for a rotation $g$ such that $g\pmb{w}_{k}= \ee_3$ and two other axes $\pmb{w}_{k}$ turn to two orthogonal axes $ \bb_i$ for $i=1,\dotsc,n$ (exists for $n$ even) (the identity rotation works as well since $\pmb{w}_{4}$, $\pmb{w}_{10}$, $\pmb{w}_{12}$ are colinear to $\ee_1,\ee_2,\ee_3$ ($\ee_2=\bb_i$ for some $i$ for $n$ even));
			\item $[\DD_{d_3}^z]$ for a rotation $g$ such that $g\vv_j= \ee_3$ and three other axes $\pmb{w}_j$ turn to three $ \bb_i$. However, there is no three coplanar $\pmb{w}_k$ seperated by an angle of $\frac{\pi}{3}$;
			\item $[\DD_{d_5}^z]$ for a rotation $g$ such that $g\ee_i= \ee_3$ and five axes $\pmb{w}_i$ turn to five $\bb_i$. However, there is no five coplanar $\pmb{w}_k$ seperated by angle of $\frac{\pi}{5}$.
		\end{itemize}
	\end{proof}
%

\begin{lem}\label{lem:DnzwithO(2)}
  For any integer $n\geq 2$, we have
  \begin{align*}
    [\DD_{n}^{z}] \circledcirc [\SO(2) \oplus \ZZ_2^c] & =\set{[\1],[\ZZ_2^-],[\ZZ_n]} \\
    [\DD_{n}^{z}] \circledcirc [\OO(2) \oplus \ZZ_2^c] & =
    \set{[\1],[\ZZ_2^-],[\DD_{d_2}^z],[\Dnz]}
  \end{align*}
\end{lem}

\begin{proof}
By theorem \ref{thm:main} we deduce that the classes in the clips $[\DD_{n}^{z}] \circledcirc [\SO(2) \oplus \ZZ_2^c]$ are among the following list
		\begin{equation*}
		\set{[\1],[\ZZ_2^-],[\ZZ_n]}.
		\end{equation*} 
	All the classes of the above list can be realized by a rotation $g$ using the union $ \ZZ_n\cap g\SO(2)g^{-1} \cup (-(\gamma\ZZ_n\cap g\SO(2)g^{-1}))$ \eqref{eq:Inter_Type_I_et_III} where $\SO(2)$ consists of all the rotations around $\ee_3$ . Indeed, we get:
	\begin{itemize}
		\item $[\ZZ_2^-]$ for a rotation $g$ such that $g\ee_3=\bb_i$ for $i=1,\dotsc,n$,
		\item $[\ZZ_n]$ for the identity rotation for instance.
	\end{itemize}
As for the classes in $[\DD_{n}^{z}] \circledcirc [\OO(2) \oplus \ZZ_2^c]$, by theorem \ref{thm:main} we know that such classes are among the following list
		\begin{equation*}
	\set{[\1],[\ZZ_{d_2}],[\ZZ_2^-],[\DD_{d_2}^z],[\DD_n^z]}.
	\end{equation*}
	By the same reasoning, the above classes can be all realized except for $[\ZZ_{d_2}]$ using \eqref{eq:Inter_Type_I_et_III}
	\begin{equation*}
	\ZZ_n\cap g\OO(2)g^{-1} \cup (-(\gamma\ZZ_n\cap g\OO(2)g^{-1}))
	\end{equation*}
	where $\gamma\ZZ_n=\set{\vR(\bb_i,\pi),i=1,\dotsc,n}$ and $g\OO(2)g^{-1}=\set{r(g\ee_3,\theta),r(g\bb,\pi),\bb\in (xy)\text{ plan}}$. Indeed, we get
	\begin{itemize}
		\item $[\ZZ_2^-]$ for a rotation $g$ such that $g\bb=\bb_i$ for $ i=1,\dotsc n$ (take for instance $g=\vR\left(\ee_1,\frac{\pi}{3}\right)\circ \vR\left(\ee_3,\frac{\pi}{2}\right)$ that turns only $\ee_2$ to $\ee_1$);
		\item $[\DD_{d_2}^z]$ for $g=\vR\left(\ee_2,\frac{\pi}{2}\right)$ for example;
		\item $[\DD_n^z]$ for the identity rotation for example;
	\end{itemize}

\end{proof}


\subsection{Clips with $\Dnd$}

First we have (see \autoref{sec:O3-subgroups})
\begin{align*}
  \Dnd & =\DD_n\cup (-\gamma \DD_n),\quad \gamma=\vR\left(\ee_3,\frac{\pi}{n}\right)
\end{align*}
where we can write
\begin{align}
   \DD_n         & =\set{\vR\left(\ee_3,\frac{2k_1\pi}{n}\right),\vR(\bb_{2l+1},\pi);\quad k_1=0,\dotsc,n-1}   \\
 \label{eq:gammaDn} -\gamma \DD_n & =\set{-\vR\left(\ee_3,\frac{(2k_2+1)\pi}{n}\right),-\vR(\bb_{2l},\pi);k_2=0,\dotsc,n-1}
\end{align}
with
\begin{equation}\label{eq:Secon_Axe_D2nm}
  \bb_1=\ee_1 \text{ and } \bb_l=\vR\left(\ee_3,\frac{\pi}{n}\right)\bb_{l-1}.
\end{equation}

\begin{lem}\label{lem:D2nhclipsZm}
  Let $m,\ n \geq 2$ be two integers and $d=\gcd(m,n)$. Then
  \begin{align*}
    [\Dnd] \circledcirc [\ZZ_m \oplus \ZZ_2^c]=
    \begin{cases}
      \set{[\1],[\ZZ_2],[\ZZ_{2}^-],[\ZZ_{2d}^-]} & \text{ if $\frac{m}{d}$ even}             \\
      \set{[\1],[\ZZ_2],[\ZZ_{2}^-],[\ZZ_{d}]}    & \text{ if $m$ even and $\frac{m}{d}$ odd} \\
      \set{[\1],[\ZZ_d]}                          & \text{ else}
    \end{cases}.
  \end{align*}
\end{lem}

\begin{proof}
  From \eqref{eq:Inter_Type_I_et_III} we have to consider intersection
  \begin{equation*}
    \Dnd \cap (g\ZZ_m g^{-1} \oplus \ZZ_2^c)=(\DD_n\cap g \ZZ_m g^{-1})\cup (-(\gamma\DD_n\cap g \ZZ_m g^{-1}))\quad \gamma=\vR\left(\ee_3,\frac{\pi}{n}\right)
  \end{equation*}
  which can always reduces to $\1$. Otherwise we have only to consider three cases:
  \begin{itemize}
    \item $g\be_3=\pm\be_3$ and we deduce intersection from lemma~\ref{lem:clips_Zm};
    \item $g\be_3=\pm\bb_{2l}$ for some $k$ and intersection reduces to $\ZZ_2^-$ for $m$ even;
    \item $g\be_3=\pm\bb_{2l+1}$ for some $k$ and intersection reduces to $\ZZ_2$ for $m$ even, and we can conclude the proof.
  \end{itemize}
\end{proof}


\begin{lem}\label{lem:D2nhclipsDm}
  Let $m,\ n \geq 2$ be two integers, $d=\gcd(m,n)$ and $d_2=\gcd(2,n)$. Then
  \begin{align*}
    [\DD_{2n}^{d}] \circledcirc [\DD_m \oplus \ZZ_2^c]=
    \begin{cases}
      \set{[\1],[\ZZ_2],[\DD_{d_2}],[\ZZ_2^-],[\ZZ_{2d}^-],[\DD_2^z],[\DD_{2d}^d]}     & \text{ if $\frac{m}{d}$ is even}                \\
      \set{[\1],[\ZZ_2],[\DD_2],[\ZZ_2^-],[\DD_2^z],[\ZZ_d],[\DD_{d}],[\DD_{d}^z]} & \text{ if $\frac{m}{d}$ is odd and $m$ is even} \\
      \set{[\1],[\ZZ_2],[\ZZ_2^-],[\ZZ_d],[\DD_{d}],[\DD_{d}^z]}                   & \text{ else}
    \end{cases}.
  \end{align*}
\end{lem}

\begin{proof}
	We apply theorem \ref{thm:main} with $\Gamma_+=\DD_n$ and $\tilde{\Gamma}=\DD_{2n}$. We deduce from \cite[table 1]{Olive2019}
	\begin{equation*}
	[\DD_n]\circledcirc [\DD_m]=\set{[\1],[\ZZ_2],[\DD_2] (\text{if $m$ and $n$ even}),[\ZZ_{d}],[\DD_d]} 
	\end{equation*}
	and
	\begin{equation*}
	[\DD_{2n}]\circledcirc [\DD_m]=\begin{cases}
	\set{[\1],[\ZZ_2],[\DD_2],[\ZZ_{2d}],[\DD_{2d}]} &\text{if $\frac{m}{d}$ even}\\
	\set{[\1],[\ZZ_2],[\DD_2],[\ZZ_{d}],[\DD_{d}]}  & \text{if $m$ even and $\frac{m}{d}$ odd}\\
	\set{[\1],[\ZZ_2],[\ZZ_{d}],[\DD_{d}]} & \text{if $m$ odd}
	\end{cases}
	\end{equation*}
	Hence, the classes in $[\DD_n^z]\circledcirc [\DD_m\oplus \ZZ_2^c]$, corresponding to type I subgroups, are among the following list 
	\begin{equation*}
	([\DD_n]\circledcirc [\DD_m])\cap ([\DD_{2n}]\circledcirc[\DD_m])=
	\begin{cases}
	\set{[\1],[\ZZ_{2}],[\DD_2]} & \text{if $\frac{m}{d}$ even}\\
	\set{[\1],[\ZZ_{2}],[\DD_2],[\ZZ_d],[\DD_d]} & \text{if $m$ even and $\frac{m}{d}$ odd}\\
	\set{[\1],[\ZZ_{2}],[\ZZ_d],[\DD_d]} & \text{if $m$ odd}
	\end{cases}
	\end{equation*}
	and the classes corresponding to type III subgroups are among
	\begin{equation*}
	\begin{cases}
	\set{[\ZZ_2^-],[\DD_{2}^z],[\ZZ_{2d}^-],[\DD_{2d}^z]} & \text{if $\frac{m}{d}$ even}\\
	\set{[\ZZ_2^-],[\DD_{2}^z],[\DD_d^z]} & \text{if $m$ even and $\frac{m}{d}$ odd}\\
	\set{[\ZZ_2^-],[\DD_d^z]} & \text{if $m$ odd}
	\end{cases}.
	\end{equation*}
	We can check that all the possibilities can occur by using \eqref{eq:Inter_Type_I_et_III}:
	\begin{equation*}
	(\DD_n\cap g \DD_m g^{-1})\cup (-(\gamma\DD_n\cap g \DD_m g^{-1}))
	\end{equation*}
where $\DD_n$ and $\gamma \DD_n$ are given in \eqref{eq:Dn} and \eqref{eq:gammaDn}.
	We obtain
	\begin{itemize}
		\item $[\DD_{2d}^d]$ if $m/d$ is even and $[\DD_d]$ otherwise for a rotation $g$ such that $g\ee_3=\ee_3$ and $g\bb_i=\pm\bb_{2l+1}$ for some $l$;\\
		\item $[\DD_{2d}^d]$ if $m/d$ is even and $[\DD_d^z]$ otherwise for a rotation $g$ such that  $g\ee_3= \ee_3$ and $g\bb_i=\pm \bb_{2l}$ for some $l$;\\
		\item $[\ZZ_{2d}^-]$ if $m/d$ is even and $[\ZZ_d]$ otherwise for a rotation $g$ such that $g\ee_3=\ee_3$ only; \\
		 \item $\begin{cases}
		 [\DD_{2}] &\text{ if } m \text{ and } n \text{ even }       \\
		 [\DD_{2}^z] &\text{ if } m \text{ even and } n \text{ odd } \\
		 [\ZZ_{2}] &\text{ if } m \text{ odd and } n \text{ even }   \\
		 [\ZZ_{2}^-] &\text{ if } m \text{ and } n \text{ odds }
		 \end{cases}$ for a rotation $g$ such that $g\ee_3=\pm \bb_{2l+1}$ for some $l$,\\
		  \item $[\ZZ_2]$ for a rotation $g$ such that $g\ee_3=\pm \bb_{2l+1}$ for some $l$;\\
		  \item $\begin{cases}
		  [\DD_{2}^z] &\text{ if } m \text{ even }                  \\
		  [\ZZ_{2}] &\text{ if } m \text{ odd and } n \text{ even } \\
		  [\ZZ_{2}^-] &\text{ if } m \text{ and } n \text{ odds }
		  \end{cases}$ for a rotation $g$ such that $g\ee_3=\bb_{2l}$ and $g\bb_i=\ee_3$;\\
		  \item $[\ZZ_2^-]$ if $m$ is even for a rotation $g$ such that $g\ee_3=\bb_{2l}$.
	\end{itemize}
\end{proof}
\begin{lem}
  For any integer $n\geq 2$ and $d_k=\text{gcd}(n,k)$ for $k=2,3$, we have
  \begin{align*}
    [\Dnd]\circledcirc [\tetra\oplus\ZZ_2^c]           
               =\set{[\1],[\ZZ_2],[\ZZ_{d_3}],[\DD_{d_2}],[\ZZ_2^-],[\DD_2^z]}.                                                                    
   \end{align*}

\end{lem}

\begin{proof}
	We deduce from \cite[table 1]{Olive2019} and theorem \ref{thm:main} that the classes in $[\DD_{2n}^d]\circledcirc [\tetra\oplus \ZZ_2^c]$, corresponding to type I subgroups, are among the following list
\begin{equation*}
([\DD_n]\circledcirc [\tetra])\cap ([\DD_{2n}]\circledcirc[\tetra])=\set{[\1],[\ZZ_{2}],[\ZZ_{d_3}],[\DD_{d_2}]}
\end{equation*}
and the classes corresponding to type III subgroups are among
\begin{equation*}
\set{[\ZZ_2^-],[\DD_{2}^z]}.
\end{equation*}
We can check that all the possibilities can occur by using \eqref{eq:Inter_Type_I_et_III}:
\begin{equation*}
(\DD_n\cap g \tetra g^{-1})\cup (-(\gamma\DD_n\cap g \tetra g^{-1})),
\end{equation*}
where $\DD_n$ and $ \gamma\DD_n$ are given in \eqref{eq:Dn} and \eqref{eq:gammaDn} and $\tetra$ in \eqref{eq:Decomposition_tetra}.
We get
\begin{itemize}
	\item $[\ZZ_{d_3}]$ for a rotation $g$ such that $g\pmb{s}_{t_j}=\pm \ee_3$;
	\item $[\ZZ_{2}]$ for a rotation $g$ such that $g\ee_i= \bb_{2l+1}$ for $i=1,\dotsc,3$ (take for instance $g=\vR\left(\ee_1,\frac{\pi}{3}\right)\circ\vR\left(\ee_3,\frac{\pi}{2l}\right)$);
	\item $[\ZZ_2^-]$ for a rotation $g$ such that $g\ee_i= \bb_{2l}$ for $ i=1,\dotsc 3$;
	\item $[\DD_{d_2}]$ for a rotation $g$ such that $g\ee_i= \ee_3$ for $i=1,2,3$ and the two remaining $\ee_i$ turn to two $\bb_{2l+1}$ for some $l$;
	\item $[\DD_{d_2}^z]$ for a rotation $g$ such that $g\ee_i= \ee_3$ for $i=1,2,3$ and the two remaining $\ee_i$ turn to two $\bb_{2l}$ for some $l$.
\end{itemize}
\end{proof}

\begin{lem}
Let $n$ be any integer and $d_k=\text{gcd}(n,k)$ for $k=2,3$, we have
	\begin{align*}
    [\DD_{2n}^{d}] \circledcirc [\octa \oplus \ZZ_2^c] 
               & =
    \begin{cases}
      \set{\mathsf{L}_{\octa},[\ZZ_4^-],[\ZZ_4],[\DD_4],[\DD_4^z]} & \text{if $4|n$}                      \\
      \set{\mathsf{L}_{\octa},[\ZZ_4^-],[\DD_4^d]}       & \text{if $n$ is even and $4\nmid n$} \\
      \set{\mathsf{L}_{\octa}}                                   & \text{if $n$ is odd}
    \end{cases}                                                                                                                                                  \\
               & \mathsf{L}_{\octa}:=[\1],[\ZZ_2],[\DD_{d_2}],[\ZZ_2^-],[\DD_2^z],[\ZZ_{d_3}],[\DD_{d_3}],[\DD_{d_3}^z]                                         
    \end{align*}
 \end{lem}
 \begin{proof}
 	We deduce from \cite[table 1]{Olive2019}
 	\begin{equation*}
 	[\DD_n]\circledcirc [\octa]=\set{[\1],[\ZZ_2],[\ZZ_{d_3}],[\ZZ_{d_4}],[\DD_{d_2}],[\DD_{d_3}],[\DD_{d_4}]}
 	\end{equation*} 
 	where $d_4=gcd(4,n)$ and 
 	\begin{equation*}
 	[\DD_{2n}]\circledcirc[\octa]=\set{[\1],[\ZZ_2],[\ZZ_{d_3}],[\ZZ_{2d_2}],[\DD_2],[\DD_{d_3}],[\DD_{2d_2}]}.
 	\end{equation*}
 	We deduce, by theorem \ref{thm:main}, that the classes in $[\DD_{2n}^d]\circledcirc [\octa\oplus \ZZ_2^c]$, corresponding to type I subgroups, are among the following list
 \begin{equation*}
 ([\DD_n]\circledcirc [\octa])\cap ([\DD_{2n}]\circledcirc[\octa])=\set{[\1],[\ZZ_{2}],[\ZZ_{d_3}],[\DD_{d_2}],[\DD_{d_3}],[\ZZ_4](\text{ if $4\mid n$}),[\DD_4](\text{ if $4\mid n$})}
 \end{equation*}
 and the classes corresponding to type III subgroups are among
 \begin{equation*}
 \set{[\ZZ_2^-],[\ZZ_4^-](\text{ if $n$ is even}),[\DD_{2}^z],[\DD_{d_3}^z],[\DD_4^z](\text{ if $4\mid n$}),[\DD_4^d](\text{ if $n$ is even and $4\nmid n$})}.
 \end{equation*}
 We can check that all the possibilities can occur by using \eqref{eq:Inter_Type_I_et_III}:
 \begin{equation*}
 (\DD_n\cap g \octa g^{-1})\cup (-(\gamma\DD_n\cap g \octa g^{-1}))
 \end{equation*}
 where  $\DD_n$ and $ \gamma\DD_n$ are given in \eqref{eq:Dn} and \eqref{eq:gammaDn} and $\octa$ in \eqref{eq:Decomposition_Cube}.
 We get

 	\begin{itemize}
 	\item $[\ZZ_2]$ for a rotation $g$ such that $g\pmb{a}_{c_k}= \bb_{2l+1}$ for $ i=1,\dotsc n$;
 	\item $[\DD_{d_2}]$ for a rotation $g$ such that $g\pmb{a}_{c_k}= \ee_3$ and two other edge axes $\pmb{a}_{c_k}$ turn to two orthogonal axes $ \bb_{2l+1}$ (exists for $n$ even) for some $l$;
 	\item $[\ZZ_{2}^-]$ for a rotation $g$ such that $g\pmb{a}_{c_k}=\bb_{2l}$;
 	\item $[\DD_2^z]$ for a rotation $g$ such that $g\pmb{a}_{c_k}=\ \bb_{2l+1}$ and two other edge axes $\pmb{a}_{c_k}$ turn to two orthogonal axes $ \bb_{2l}$ for some $l$;
 	\item $[\ZZ_{d_3}]$ for a rotation $g$ such that $g\pmb{s}_{t_j}= \ee_3$;
 	\item $[\DD_{d_3}]$ or $[\DD_{d_3}^z]$ for a rotation $g$ such that $g\pmb{s}_{t_j}=\ee_3$ and three other edge axes $\pmb{a}_{c_k}$ turn to three secondary axis of $\DD_n$ (either three $\bb_{2l+1}$ or three $\bb_{2l}$) (one can take for instance $g=\vR\left(\ee_3,\frac{\pi}{4}\right)\circ \vR\left(<1,-1,0>,\arccos\left(\sfrac{1}{\sqrt{3}}\right)\right)$ to get $\DD_{d_3}$);
 	\item $[\ZZ_{4}]$ for a rotation $g$ such that $g\ee_i= \ee_3$ for $i=1,\dotsc,3$ (when $4\mid n$);
 	\item $[\ZZ_4^-]$ for a rotation $g$ such that $g\ee_i= \ee_3$ for $i=1,\dotsc,3$ (when $n$ is even and $4\nmid n$);
 	\item $[\DD_4]$ or $[\DD_{4}^z]$ for a rotation $g$ such that $g\ee_i= \ee_3$ with $4\mid n$ and two edge axes $\pmb{a}_{c_k}$ together with the remaining two $\ee_i$ turn to four secondary axis of $\DD_n$ (either four $\bb_{2l+1}$ or four $\bb_{2l}$);
 	\item $[\DD_{4}^d]$ for the identity rotation for instance.
 \end{itemize}
\end{proof}

 \begin{lem}
 	For any integer $n\geq 2$ and $d_k=\text{gcd}(n,k)$ for $k=2,3,5$, we have
 	\begin{align*}
 	 [\Dnd]\circledcirc [\ico\oplus\ZZ_2^c]             
 	& =\left\lbrace [\1],[\ZZ_2^-],[\ZZ_2],[\DD_{d_2}],[\DD_2^z],[\ZZ_{d_3}], [\ZZ_{d_5}] \right\rbrace
 	\end{align*}
 \end{lem} 

\begin{proof}
  	We deduce from \cite[table 1]{Olive2019} and theorem \ref{thm:main} that the classes in $[\DD_{2n}^d]\circledcirc [\ico\oplus \ZZ_2^c]$, corresponding to type I subgroups, are among the following list
  \begin{equation*}
  ([\DD_n]\circledcirc [\ico])\cap ([\DD_{2n}]\circledcirc[\ico])=\set{[\1],[\ZZ_{2}],[\DD_{d_2}],[\ZZ_{d_3}],[\DD_{d_3}^z],[\ZZ_{d_5}],[\DD_{d_5}^z]}
  \end{equation*}
  and the classes corresponding to type III subgroups are among
  \begin{equation*}
  \set{[\ZZ_2^-],[\DD_{2}^z] ,[\DD_{d_3}^z],[\DD_{d_5}^z]}.
  \end{equation*}
  We can check that all the possibilities can occur except $[\DD_{d_3}], \ [\DD_{d_3}^z],\ [\DD_{d_5}],\ [\DD_{d_5}^z]$ for the same argument as in lemma \ref{lem:Dnz-with_ico} by using \eqref{eq:Inter_Type_I_et_III}:
  \begin{equation*}
  (\DD_n\cap g \ico g^{-1})\cup (-(\gamma\DD_n\cap g \ico g^{-1}))
  \end{equation*}
   where  $\DD_n$ and $ \gamma\DD_n$ are given in \eqref{eq:Dn} and \eqref{eq:gammaDn} and $\ico$ in \eqref{eq:Decomposition_Ico}.
  We get
  \begin{itemize}
  	
  	\item $[\ZZ_2^-]$ for a rotation $g$ such that $g\pmb{w}_{k}=\bb_{2l}$ for $k=1,\dotsc, 15$;
  	\item $[\ZZ_{2}]$ for a rotation $g$ such that $g\pmb{w}_{k}= \bb_{2l+1}$ for $k=1,\dotsc, 15$;
  	\item $[\DD_{d_2}]$ for a rotation $g$ such that $g\pmb{w}_{k}= \ee_3$ and two other axes $\pmb{w}_{k}$ turn to two orthogonal axes $\bb_{2l+1}$;
  	\item $[\DD_{d_2}^z]$ for a rotation $g$ such that $g\pmb{w}_{k}= \ee_3$ and two other axes $\pmb{w}_{k}$ turn to two orthogonal axes $ \bb_{2l}$;
  	\item $[\ZZ_{d_3}]$ for a rotation $g$ such that $g\vv_{j}= \ee_3$ for $j=1,\dotsc,10$;
  	\item $[\ZZ_{d_5}]$ for a rotation $g$ such that $g\uu_i= \ee_3$ for $i=1,\dotsc,6$;
 
  \end{itemize}
\end{proof}

Finally, we have

\begin{lem}
  For any integer $n\geq 2$ we have
  \begin{align*}
    [\DD_{2n}^{d}] \circledcirc [\SO(2) \oplus \ZZ_2^c] & =\set{[\1],[\ZZ_2],[\ZZ_2^-],[\ZZ_{2n}^-]}               \\
    [\DD_{2n}^{d}] \circledcirc [\OO(2) \oplus \ZZ_2^c] & =\set{[\1],[\ZZ_2],[\ZZ_2^-],[\DD_{d_2}],[\DD_2^z],[\Dnd]}.
  \end{align*}
\end{lem}

\begin{proof}
	By theorem \ref{thm:main} we deduce that the classes in the clips $[\Dnd] \circledcirc [\SO(2) \oplus \ZZ_2^c]$ are among the following list
		\begin{equation*}
		\set{[\1],[\ZZ_2],[\ZZ_2^-],[\ZZ_{2n}^-]}.
		\end{equation*} 
		All the classes of the above list can be realized by a rotation $g$ using the union $$ \DD_n\cap g\SO(2)g^{-1} \cup (-(\gamma\DD_n\cap g\SO(2)g^{-1}))\eqref{eq:Inter_Type_I_et_III}.$$ Indeed, we get:
		\begin{itemize}
			\item $[\ZZ_2]$ for a rotation $g$ such that $g\ee_3=\bb_{2l+1}$;
			\item $[\ZZ_2^-]$ for a rotation $g$ such that $g\ee_3=\bb_{2l}$;
			\item $[\ZZ_{2n}^-]$ for the identity rotation for instance.
		\end{itemize}
		As for the classes in $[\Dnd] \circledcirc [\OO(2) \oplus \ZZ_2^c]$, by theorem \ref{thm:main} we know that such classes are among the following list
		\begin{equation*}
		\set{[\1],[\ZZ_{2}],[\ZZ_2^-],[\DD_{d_2}],[\DD_{2n}^d],[\DD_{2}^z]}.
		\end{equation*}
		By the same reasoning, the above classes can be all realized by a rotation $g$ using \eqref{eq:Inter_Type_I_et_III}
		\begin{equation*}
		\DD_n\cap g\OO(2)g^{-1} \cup (-(\gamma\DD_n\cap g\OO(2)g^{-1}))
		\end{equation*}
		where $g\OO(2)g^{-1}=\set{r(g\ee_3,\theta),r(g\bb,\pi),\bb\in xy \text{ plane}}$. We get,
		
		\begin{itemize}
			\item $[\ZZ_2]$ for a rotation $g$ such that $g=\vR\left(\ee_1,\frac{\pi}{3}\right)\circ \vR\left(\ee_2,\frac{\pi}{2}\right)$ that turns only $\ee_3$ to $\ee_1$ for instance;
			\item $[\ZZ_2^-]$ for a rotation $g$ such that $g\bb=\bb_{2k}$ for some $k$ and some $\bb\in (xy)$-plane and that doesn't turn $\ee_3$ to the other axes of $\DD_{2n}$ (take for instance $g=\vR\left(\bb_{2k},\frac{\pi}{3}\right)\circ \vR\left(\ee_3,\alpha\right)$ where $\alpha$ is the angle between $\bb$ and $\bb_{2k}$;
			\item $[\DD_{d_2}]$ for a rotation $g=\vR\left(\ee_2,\frac{\pi}{2}\right)$ for instance (for $n$ even there exists $k$ such that $\bb_{2k+1}=\ee_2$);
			\item $[\DD_2^z]$ for a rotation $g=\vR\left(\vv,\frac{\pi}{2}\right)$ where $\vv$ is a secondary axis of $\DD_{2n}$ orthogonal to $\ee_3$ and $\bb_{2k}$ for instance;
			\item $[\DD_{2n}^d]$ for the identity rotation for example.
			
		\end{itemize}
\end{proof}

\subsection{Clips with $\octa^-$}

First, we introduce a useful decomposition of the subgroup $\octa^-$ (see \autoref{sec:O3-subgroups}), constructed from the couple of index 2 $(\tetra,\octa)$
\begin{equation*}
	\octa^-=\tetra \cup (-\gamma\tetra),\quad \gamma=\vR\left(\ee_1,\frac{\pi}{2}\right)
\end{equation*}
where
\begin{equation}
	\tetra=\bigcup_{i=1}^3\ZZ_2^{\ee_i}\bigcup_{j=1}^4 \ZZ_3^{\pmb{s}_{t_j}},
\end{equation}
\begin{equation}\label{eq:gammaT}
	\gamma\tetra=\set{\vR\left(\pmb{a}_{c_k},\pi\right),\vR\left(\ee_i,\frac{\pi}{2}\right),\vR\left(\ee_i,\frac{3\pi}{2}\right)}.
\end{equation}

\begin{lem}\label{lem:octaclipsZm}
  Let $m\geq 2$ be an integer and $d_k'=\gcd(k,m)$ ($k=2,3$). Then we have
  \begin{align*}
    [\octa^-] \circledcirc [\ZZ_m \oplus \ZZ_2^c]=
    \begin{cases}
      \set{[\1],[\ZZ_{2}^-],[\ZZ_{d_3'}],[\ZZ_4^-]}       & \text{if $4|m$}  \\
      \set{[\1],[\ZZ_{d'_2}],[\ZZ_{d'_2}^-],[\ZZ_{d_3'}]} & \text{otherwise}
    \end{cases}.
  \end{align*}
\end{lem}

\begin{proof}
	From \eqref{eq:Inter_Type_I_et_III} we have to consider intersection
	\begin{equation*}
		\octa^- \cap (g\ZZ_m g^{-1} \oplus \ZZ_2^c)=(\tetra\cap g \ZZ_m g^{-1})\cup (-(\gamma\tetra\cap g \ZZ_m g^{-1}))\quad \gamma=\vR\left(\ee_1,\frac{\pi}{n}\right)
	\end{equation*}
	which can always reduces to $\1$. Otherwise we have only to consider three cases:
	\begin{itemize}
		\item $g\ee_3=\pm\ee_i$ and we get $[\ZZ_4^-]$ if $4\mid m$ and $[\ZZ_{d_2'}]$ if not ;
		\item $g\ee_3=\pm\pmb{s}_{t_j}$ and we get $[\ZZ_{d_3'}]$ ;
		\item $g\ee_3=\pm\pmb{a}_{c_k}$ and we get $[\ZZ_{d_2'}^-]$.
	\end{itemize}
\end{proof}


\begin{lem}
  Let $m\geq 2$ be an integer. We have
  \begin{align*}
    [\octa^-] \circledcirc [\DD_m \oplus \ZZ_2^c]=
    \begin{cases}
      \set{[\1],[\ZZ_2],[\ZZ_{2}^-],[\ZZ_{d_3'}],[\DD_{d_3'}^z],[\ZZ_4^-],[\DD_2^z],[\DD_4^d]} & \text{if $4|m$}                      \\
      \set{[\1],[\ZZ_2],[\ZZ_{2}^-],[\ZZ_{d_3'}],[\DD_{d_3'}^z],[\DD_2],[\DD_2^z]}             & \text{if $m$ is even and $4\nmid m$} \\
      \set{[\1],[\ZZ_2],[\ZZ_2^-],[\ZZ_{d_3'}],[\DD_{d_3'}^z]}                                 & \text{if $m$ is odd}
    \end{cases}.
  \end{align*}
\end{lem}

\begin{proof}
	We apply theorem \ref{thm:main} with $\Gamma_+=\tetra$ and $\tilde{\Gamma}=\octa$. We deduce from \cite[table 1]{Olive2019}
	\begin{equation*}
	[\tetra]\circledcirc [\DD_m]=\set{[\1],[\ZZ_{d_2'}],[\ZZ_{d_3'}],[\DD_{d_2'}]} 
	\end{equation*}
	and
	\begin{equation*}
	[\octa]\circledcirc [\DD_m]=\set{[\1],[\ZZ_2],[\ZZ_{d_3'}],[\ZZ_{d_4'}],[\DD_{d_2'}],[\DD_{d_3'}],[\DD_{d_4'}]}.
	\end{equation*}
	Hence, the classes in $[\octa^-]\circledcirc [\DD_m\oplus \ZZ_2^c]$, corresponding to type I subgroups, are among the following list 
	\begin{equation*}
	([\tetra]\circledcirc [\DD_m])\cap ([\octa]\circledcirc[\DD_m])=
	\set{[\1],[\ZZ_2],[\ZZ_{d_3'}],[\DD_{d_2'}]}
	\end{equation*}
	and the classes corresponding to type III subgroups are among
	\begin{equation*}
	\begin{cases}
	\set{[\ZZ_2^-],[\DD_{4}^d],[\ZZ_{4}^-],[\DD_{2}^z],[\DD_{d_3'}^z]} & \text{if $4\mid m$}\\
	\set{[\ZZ_2^-],[\DD_{2}^z],[\DD_{d_3'}]} & \text{if $m$ even and $4\nmid m$}\\
	\set{[\ZZ_2^-],[\DD_{d_3'}^z]} & \text{if $m$ odd}
	\end{cases}.
	\end{equation*}
 We can check that all the possibilities can occur by using \eqref{eq:Inter_Type_I_et_III}, except $[\DD_2]$ when $4\mid m$ since it becomes $\DD_4^d$ in this case. Consider
	\begin{equation*}
	(\tetra\cap g \DD_m g^{-1})\cup (-(\gamma\tetra\cap g \DD_m g^{-1}))
	\end{equation*}
	where $\gamma \tetra$ is given by \eqref{eq:gammaT}.
	We obtain
	\begin{itemize}
		\item $[\ZZ_2]$ for a rotation $g$ such that $g\bb_i=\ee_j$ for $i=1,\dotsc,m$ and $j=1,\dots,3$ (take for instance $g=\vR\left(\ee_2,\frac{\pi}{3}\right)\circ \vR\left(\ee_3,\frac{\pi}{2}\right)$);
		\item $[\ZZ_2^-]$ for a rotation $g$ such that $g\bb_i=\pmb{a}_{c_k}$ for $i=1,\dotsc,m$ (take for instance $g=\vR\left(\pmb{a}_{c_k},\frac{\pi}{3}\right)\circ \vR\left(\ee_3,\frac{\pi}{4}\right)$);
		\item $[\ZZ_{d_3'}]$ for a rotation $g$ such that $g\ee_3=\pmb{s}_{t_j}$;
		\item $[\DD_{d_3'}^z]$ for a rotation $g$ such that $g\ee_3=\pmb{s}_{t_j}$ and three secondary axis $\bb_i$ of $\DD_m$ turn to three $\pmb{a}_{c_k}$;
		\item $[\DD_2]$ for a rotation $g$ such that $g\ee_3=\ee_i$ and two orthogonal secondary axis of $\DD_m$ turn to the two remaining $\ee_i$ and that is possible when $m$ is even and $4\nmid m$ since if $4\mid m$ we get $[\DD_4^d]$;
		\item $[\DD_2^z]$ for a rotation $g$ such that $g\ee_3=\ee_i$ and two orthogonal secondary axis of $\DD_m$ turn to orthogonal $\pmb{a}_{c_k}$ and this is possible if $m$ is even;
		\item $[\ZZ_4^-]$ for a rotation $g$ such that $g\ee_3=\ee_i$ for $i=1,\dotsc,3$ and if $4\mid m$.
	
	\end{itemize}
\end{proof}

\begin{lem}
  We have
  \begin{align*}
      [\octa^-] \circledcirc [\tetra \oplus \ZZ_2^c]=\set{[\1],[\ZZ_2],[\ZZ_2^-],[\DD_2],[\DD_2^z],[\ZZ_3],[\tetra]}.                  
  \end{align*}
\end{lem}

\begin{proof}
		We deduce from \cite[table 1]{Olive2019} and theorem \ref{thm:main} that the classes in $[\octa^-]\circledcirc [\tetra\oplus \ZZ_2^c]$, corresponding to type I subgroups, are among the following list
	\begin{equation*}
		([\tetra]\circledcirc [\tetra])\cap ([\octa]\circledcirc[\tetra])=\set{[\1],[\ZZ_{2}],[\ZZ_{3}],[\DD_{2}],[\tetra]}
	\end{equation*}
	and the classes corresponding to type III subgroups are among
	\begin{equation*}
		\set{[\ZZ_2^-],[\DD_{2}^z]}.
	\end{equation*}
	We can check that all the possibilities can occur by using \eqref{eq:Inter_Type_I_et_III}:
	\begin{equation*}
		(\tetra\cap g \tetra g^{-1})\cup (-(\gamma\tetra\cap g \tetra g^{-1})),
	\end{equation*}
	where $ \gamma\tetra$ is given by \eqref{eq:gammaT} and $\tetra$ in \eqref{eq:Decomposition_tetra}.
	We get
	\begin{itemize}
		\item $[\ZZ_{3}]$ for a rotation $g$ such that $g\pmb{s}_{t_j}=\pmb{s}_{t_l}$;
		\item $[\DD_{2}]$ for a rotation $g=\vR\left(\ee_2,\frac{\pi}{2}\right)$ for instance;
		\item $[\ZZ_{2}]$ for a rotation $g=\vR\left(\ee_3,\frac{\pi}{3}\right)\circ\vR\left(\ee_2,\frac{\pi}{2}\right)$ for instance;
		\item $[\DD_{2}^z]$ for a rotation $g$ such that $g\ee_i= \ee_j$ for $i=1,\dotsc,3$ and the two remaining $\ee_i$ turn to two $\pmb{a}_{c_k}$ (for instance $g=\vR\left(\ee_3,\frac{\pi}{4}\right)$);
		\item $[\ZZ_2^-]$ for a rotation $g$ such that $g\ee_i= \pmb{a}_{c_k}$ for $ i=1,\dotsc 3$ (for instance $g=\vR\left(\pmb{a}_{c_1},\frac{\pi}{3}\right)\circ\vR\left(\ee_3,\frac{\pi}{2}\right)$);
		\item $[\tetra]$ for the identity rotation for instance.
	\end{itemize}
\end{proof}

\begin{lem}
	We have
	\begin{equation*}
		 [\octa^-] \circledcirc [\octa \oplus \ZZ_2^c]=\set{[\1],[\ZZ_2],[\ZZ_2^-],[\ZZ_3],[\ZZ_4^-],[\DD_2^z],[\DD_3^z],[\DD_4^d],[\octa^-]}.
	\end{equation*}
\end{lem}

\begin{proof}
We deduce, by theorem \ref{thm:main}, that the classes in $[\octa^-]\circledcirc [\octa\oplus \ZZ_2^c]$, corresponding to type I subgroups, are among the following list
\begin{equation*}
	([\tetra]\circledcirc [\octa])\cap ([\octa]\circledcirc[\octa])=\set{[\1],[\ZZ_{2}],[\ZZ_{3}],[\DD_{2}],[\tetra]}
\end{equation*}
and the classes corresponding to type III subgroups are among
\begin{equation*}
	\set{[\ZZ_2^-],[\ZZ_4^-],[\DD_{2}^z],[\DD_{3}^z],[\DD_4^d],[\octa^-]}.
\end{equation*}
We can check that all the possibilities can occur by using \eqref{eq:Inter_Type_I_et_III} except $[\DD_2]$ since to get $[\DD_2]$ we need a rotation $g$ such that $g\pmb{a}_{c_k}=\ee_i$ and $g\ee_i=\ee_i$ which will give $[\ZZ_4^-]$.
We have
\begin{equation*}
	(\tetra\cap g \octa g^{-1})\cup (-(\gamma\tetra\cap g \octa g^{-1}))
\end{equation*}
where  $\gamma\tetra$ is given by \eqref{eq:gammaT} and $\octa$ in \eqref{eq:Decomposition_Cube}:

\begin{itemize}
	\item $[\ZZ_2]$ for a rotation $g$ such that $g\pmb{a}_{c_k}= \ee_i$ for $ i=1,\dotsc, 3$ (take for instance $g=\vR\left(\ee_1,\frac{\pi}{3}\right)\circ \vR\left(\ee_3,-\frac{\pi}{4}\right)$);
	\item $[\ZZ_{2}^-]$ for a rotation $g=\vR\left(\pmb{a}_{c_1},\frac{\pi}{3}\right)\circ \vR\left(\ee_3,\frac{\pi}{2}\right)$ ;
	\item $[\DD_2^z]$ for a rotation $g=\vR\left(\ee_3,\frac{\pi}{2} \right)$ for example;
	\item $[\ZZ_{3}]$ for a rotation $g$ such that $g\pmb{s}_{t_j}= \pmb{s}_{t_l}$;
	\item $[\DD_{d_3}^z]$ for a rotation $g$ such that $g\pmb{s}_{t_j}=\pmb{s}_{t_l}$ and three edge axes $\pmb{a}_{c_k}$ turn to three $\pmb{a}_{c_k}$;
	\item $[\ZZ_4^-]$ for a rotation $g$ such that $g\ee_i=\ee_j$ for some $i,j=1,\dotsc,3$ (take for instance $g=\vR\left(\ee_3,\frac{\pi}{3}\right)$);
	\item $[\DD_{4}^d]$ for a rotation $g=\vR\left(\ee_3,\frac{\pi}{2}\right)$ for example;
	\item $[\octa^-]$ for the identity rotation for instance.
\end{itemize}	
\end{proof}

\begin{lem}
	We have
	\begin{equation*}
		[\octa^-] \circledcirc
		[\ico \oplus \ZZ_2^c]=\set{[\1],[\ZZ_2],[\ZZ_2^-],[\ZZ_3],[\DD_2^z],[\tetra]}.
	\end{equation*}
\end{lem}

\begin{proof}
	We deduce from \cite[table 1]{Olive2019} and theorem \ref{thm:main} that the classes in $[\octa^-]\circledcirc [\ico\oplus \ZZ_2^c]$, corresponding to type I subgroups, are among the following list
	\begin{equation*}
		([\tetra]\circledcirc [\ico])\cap ([\octa]\circledcirc[\ico])=\set{[\1],[\ZZ_{2}],[\ZZ_3],[\tetra]}
	\end{equation*}
	and the classes corresponding to type III subgroups are among
	\begin{equation*}
		\set{[\ZZ_2^-],[\DD_{2}^z] ,[\DD_{3}^z]}.
	\end{equation*}
	We can check that all the possibilities can occur except $[\DD_3^z]$ (same argument as in lemma \ref{lem:Dnz-with_ico}) by using \eqref{eq:Inter_Type_I_et_III}:
	\begin{equation*}
		(\tetra\cap g \ico g^{-1})\cup (-(\gamma\tetra\cap g \ico g^{-1}))
	\end{equation*}
	where  $ \gamma\tetra$ is given by \eqref{eq:gammaT} and $\ico$ in \eqref{eq:Decomposition_Ico}. Hence, we get
	\begin{itemize}
		
		\item $[\ZZ_2]$ for a rotation $g$ such that $g\pmb{w}_{k}=\ee_i$ for $k=1,\dotsc, 15$ and $i=1,\dotsc,3$;
		\item $[\ZZ_{2}^-]$ for a rotation $g$ such that $g\pmb{w}_{k}= \pmb{a}_{c_k}$;
		\item $[\DD_{2}^z]$ for a rotation $g$ such that $g\pmb{w}_{k}= \ee_i$ and two other axes $\pmb{w}_{k}$ turn to two orthogonal axes $ \pmb{a}_{c_k}$ (take for instance $g=\vR\left(\ee_3,\frac{\pi}{4}\right)$);
		\item $[\ZZ_{3}]$ for a rotation $g$ such that $g\vv_{j}=\pmb{s}_{t_j}$;
		\item $[\tetra]$ for the identity rotation for instance.
	\end{itemize}
\end{proof}

\begin{lem}
  We have
  \begin{align*}
      [\octa^-] \circledcirc
    [\SO(2) \oplus \ZZ_2^c]=\set{[\1],[\ZZ_3],[\ZZ_2^-],[\ZZ_4^-]}.
  \end{align*}
\end{lem}

\begin{proof}
	By theorem \ref{thm:main} we deduce that the classes in the clips $[\octa^-] \circledcirc [\SO(2) \oplus \ZZ_2^c]$ are among the following list
	\begin{equation*}
		\set{[\1],[\ZZ_2],[\ZZ_3],[\ZZ_2^-],[\ZZ_{4}^-]}.
	\end{equation*} 
	All the classes of the above list can be realized by a rotation $g$ using the union $ \tetra\cap g\SO(2)g^{-1} \cup (-(\gamma\tetra\cap g\SO(2)g^{-1}))$ \eqref{eq:Inter_Type_I_et_III} except $[\ZZ_2]$. Indeed, we get:
	\begin{itemize}
		\item $[\ZZ_3]$ for a rotation $g$ such that $g\ee_3=\pmb{s}_{t_j}$;
		\item $[\ZZ_2^-]$ for a rotation $g$ such that $g\ee_3=\pmb{a}_{c_k}$;
		\item $[\ZZ_{4}^-]$ for the identity rotation for instance.
	\end{itemize}

\end{proof}

\begin{lem}\label{lem:octaclipso(2)-}
	We have
	\begin{equation*}
		 [\octa^-] \circledcirc
		[\OO(2) \oplus \ZZ_2^c]=\set{[\1],[\ZZ_2^-],[\DD_3^z],[\DD_4^d]}.
	\end{equation*}
\end{lem}

\begin{proof}
		By theorem \ref{thm:main} we deduce that the classes in the clips $[\octa^-] \circledcirc [\OO(2) \oplus \ZZ_2^c]$ are among the following list
	\begin{equation*}
		\set{[\1],[\ZZ_2],[\DD_2],[\ZZ_2^-],[\DD_2^z],[\DD_3^z],[\DD_4^d]}.
	\end{equation*} 
	Using the union $ \tetra\cap g\OO(2)g^{-1} \cup (-(\gamma\tetra\cap g\OO(2)g^{-1}))$ \eqref{eq:Inter_Type_I_et_III} 	where $$g\OO(2)g^{-1}=\set{r(g\ee_3,\theta),r(g\bb,\pi),\bb\in xy \text{ plan}},$$ we deduce the classes that can be realized by a rotation $g$:
	\begin{itemize}
		\item $[\ZZ_2^-]$ for a rotation $g$ such that $g\ee_3=\pmb{a}_{c_k}$;
		\item $[\DD_3^z]$ for a rotation $g$ such that $g\ee_3=\pmb{s}_{t_j}$ and three axes $\bb$ of the $(xy)$-plane turn to $\pmb{a}_{c_k}$ (take for instance $g=\vR\left(<-1,1,0>,\arccos\left(\sfrac{1}{\sqrt{3}}\right)\right)$);
		\item $[\DD_{4}^d]$ for the identity rotation for instance.
	\end{itemize}
\end{proof}

\subsection{Clips with $\OO(2)^-$}

We construct $\OO(2)^-$ as follows
\begin{align*}
  \OO(2)^- & =\SO(2)\cup -(\gamma \SO(2)) \quad \text{where } \gamma=\vR(\ee_1,\pi) \\
           & =\set{\vR(\ee_3,\theta),\theta \in [0,2\pi],-\vR(\bb,\pi)}.
\end{align*}
where $\vR(\bb,\pi)$ represent the symmetry with respect to all the axes in the $xy$ plane.

\begin{lem}
  Let $m\geq 2$ be an integer and $d_2'=\gcd(2,m)$. We have
  \begin{equation*}
    [\OO(2)^-]\circledcirc [\ZZ_m \oplus \ZZ_2^c]=
    \set{ [\1],[\ZZ_m],[\ZZ_{d_2'}^-] }
  \end{equation*}
\end{lem}

\begin{proof}
  We get the result by considering the following intersection (from \eqref{eq:Inter_Type_I_et_III})
  \begin{equation*}
    \OO(2)^-\cap (g\ZZ_m g^{-1} \oplus \ZZ_2^c)=(\SO(2) \cap g \ZZ_m g^{-1})\cup (-(\gamma\SO(2)\cap g\ZZ_mg^{-1}))
  \end{equation*}
  where $\gamma=\vR(\ee_1,\pi)$ and $\gamma\SO(2)=\set{\vR(\bb,\pi)}$.
\end{proof}

The proof of the following lemma is similar to the proof of lemma \ref{lem:DnvClipsDm}.
\begin{lem}
  For any integer $m\geq 2$, we have
  \begin{align*}
    [\OO(2)^-]\circledcirc [\DD_m\oplus \ZZ_2^c]=
    \begin{cases}
      \set{[\1],[\ZZ_2^-],[\DD_2^z],[\DD_m^z]} \quad & \text{if $m$ is even} \\
      \set{[\1],[\ZZ_2],[\ZZ_2^-],[\DD_m^z]} \quad   & \text{if $m$ is odd}
    \end{cases} .
  \end{align*}
\end{lem}

\begin{proof}
		We apply theorem \ref{thm:main} with $\Gamma_+=\SO(2)$ and $\tilde{\Gamma}=\OO(2)$. We deduce from \cite[table 1]{Olive2019}
	\begin{equation*}
		[\SO(2)]\circledcirc [\DD_m]=\set{[\1],[\ZZ_{2}],[\ZZ_{m}]} 
	\end{equation*}
	and
	\begin{equation*}
		[\OO(2)]\circledcirc [\DD_m]=\set{[\1],[\ZZ_2],[\DD_{d_2'}],[\DD_{m}]}.
	\end{equation*}
	Hence, the classes in $[\OO(2)^-]\circledcirc [\DD_m\oplus \ZZ_2^c]$, corresponding to type I subgroups, are among the following list 
	\begin{equation*}
		([\SO(2)]\circledcirc [\DD_m])\cap ([\OO(2)]\circledcirc[\DD_m])=
		\set{[\1],[\ZZ_2]}
	\end{equation*}
	and the classes corresponding to type III subgroups are among
	\begin{equation*}
	\set{[\ZZ_2^-],[\DD_{d_2'}^z],[\DD_m^z]}.
	\end{equation*}
	We can check that all the possibilities can occur by using \eqref{eq:Inter_Type_I_et_III}, except $[\DD_2]$ when $4\mid m$ since it becomes $\DD_4^d$ in this case. Consider
	\begin{equation*}
		(\SO(2)\cap g \DD_m g^{-1})\cup (-(\gamma\SO(2)\cap g \DD_m g^{-1}))
	\end{equation*}
	where $\gamma\SO(2)=\set{\vR(\bb,\pi),\ \bb\in xy \text{ plane}}$.
	We obtain
	\begin{itemize}
		\item $[\ZZ_2]$ for a rotation $g$ such that $g\bb_i=\ee_3$ (possible only if $m$ is odd since if $m$ is even the second part of the intersection won't be empty);
		\item $[\ZZ_2^-]$ for a rotation $g$ such that $g\bb_i=\bb$ for some $\bb$ in the $xy$ plane;
		\item $[\DD_{d_2'}^z]$ for a rotation $g=\vR\left(\ee_2,\frac{\pi}{2}\right)$ for example;
		\item $[\DD_m^z]$ for the identity rotation for instance.
		
	\end{itemize}
\end{proof}
The argumentation for the calculation of the clips with $\OO(2)^-$ is very similar to the ones for $\DD_n^z$ exposed in \autoref{sec:withDnz}.
\begin{lem}
  We have
  \begin{align*}
     & [\OO(2)^-]\circledcirc [\tetra\oplus \ZZ_2^c]=\set{[\1],[\ZZ_2],[\ZZ_3],[\ZZ_2^-],[\DD_2^z]}.                            \\
     & [\OO(2)^-]\circledcirc [\octa\oplus \ZZ_2^c]=\set{[\1],[\ZZ_2],[\ZZ_2^-],[\DD_2^z],[\DD_3^z],[\DD_4^z]}. \\
     & [\OO(2)^-]\circledcirc [\ico\oplus \ZZ_2^c]=\set{[\1],[\ZZ_2],[\ZZ_2^-],[\DD_2^z]}.
  \end{align*}
\end{lem}

And finally we deduce the clips with $\SO(2)\oplus \ZZ_2^c$ and $\OO(2)\oplus \ZZ_2^c$ in the same way as in lemma \ref{lem:DnzwithO(2)}.

\begin{lem}
  We have
  \begin{align*}
     & [\OO(2)^-]\circledcirc [\SO(2)\oplus \ZZ_2^c]=\set{[\1],[\ZZ_2^-],[\SO(2)]}.   \\
     & [\OO(2)^-]\circledcirc [\OO(2)\oplus \ZZ_2^c]=\set{[\DD_2^z],[\OO(2)^-]}.
  \end{align*}
\end{lem}

\section{Application to Piezoelectricity}
\label{sec:piezoelectricity}

We propose here to apply clips operation to the specific case of the \emph{Piezoelectricity law}. We introduce the space of constitutive tensors occurring in the mechanical description of Piezoelectricity, which describes the electrical behavior of a material subject to mechanical stress. It is defined by a triplet of tensors given by an elasticity tensor, a Piezoelectricity tensor and a permittivity tensor. Such a space is naturally endowed with an $\OO(3)$ representation, and the finite set of isotropy classes is obtained in theorem 5.1 below.

We recall now the Piezoelectricity law, while details can be found in~\cite{Schouten1951,Landau2013,Roy99,Ieee88}. First, the mechanical state of a material is characterized by two fields of symmetric second order tensors: the stress tensor $\bsigma$ and the strain tensor $\beps$. The relation between these two fields forms the constitutive law that describes the mechanical behavior of a specific material. In linear elasticity, the relation is linear, given by
\begin{equation*}
  \bsigma=\bE:\beps
\end{equation*}
which is known as the generalized Hooke's law. Such fourth order elasticity tensor $\bE$ have the index symmetries
\begin{equation*}
  \bE_{ijkl}=\bE_{jikl}=\bE_{ijlk}=\bE_{klij}
\end{equation*}
and we define the associated space of elasticity tensors $\Ela$, which is a 21 dimensional vector space.

Similarly to the mechanical state, the electrical state of a material is described by two vector fields: the electric displacement field $\bd$ and the electric field $\be$. These two fields are related and the relation between them forms the constitutive law that describes the electrical behavior of a material. In the linear case, it is given by
\begin{equation*}
  \bd=\bS.\be
\end{equation*}
where the second order symmetric tensor $\bS$ is the \emph{permittivity} tensor. We define $\Sym$ to be the vector space of the permittivity tensors, which is of dimension 6.

Finally, the Piezoelectricity law is given by the coupled law
\begin{equation*}
  \begin{cases}
    \bsigma=\bE:\beps-\be.\bP \\
    \bd=\bP:\bsigma+\bS.\be
  \end{cases}
\end{equation*}
which involves a third order tensor $\bP$ called the \emph{Piezoelectricity} tensor, satisfying the index symmetry
\begin{equation*}
  \bP_{ijk}=\bP_{ikj}
\end{equation*}
The vector space of Piezoelectricity tensors is an 18 dimensional vector space, noted $\Piez$.

As a consequence, the linear electromechanical behavior of any homogeneous material is defined by a triplet $\mathcal{P}$ of constitutive tensors
\begin{equation*}
  \mathcal{P}:=(\bE,\bP,\bS)\in \Ela \oplus \Piez \oplus \Sym
\end{equation*}
and we define $\mathcal{P}$iez to be the space of Piezoelectricity constitutive tensors:
\begin{equation*}
  \mathcal{P}\text{iez}=\Ela \oplus \Piez \oplus \Sym.
\end{equation*}

The natural $\OO(3)$ representation on $(\bE,\bP,\bS)\in \mathcal{P}\text{iez}$ is given in any orthonormal basis by
\begin{align}\label{eq:New_Tensors}
  (\rho(g)\bE)_{ijkl} & :=g_{ip}g_{jq}g_{kr}g_{ls}E_{pqrs},\quad (\rho(g)\bP)_{ijk}:=g_{ip}g_{jq}g_{kr}P_{pqr},\quad
  (\rho(g)\bS)_{ij}:=g_{ip}g_{jq}S_{pq}
\end{align}
where $g\in\OO(3)$.

As a consequence of lemma~\ref{lem:direct_sum}, the isotropy classes $\J(\mathcal{P}\text{iez})$ can be deduced from isotropy classes $\Ela$, $\Piez$ and $\Sym$:
\begin{equation*}
  \J(\mathcal{P}\text{iez})=\left(\J(\Ela)\circledcirc \J(\Piez)\right) \circledcirc \J(\Sym).
\end{equation*}

Recall from~\cite{Olive2021} the isotropy classes of $\Piez$ (the notations and definitions of $\OO(3)$-subgroups have been moved to \autoref{sec:O3-subgroups})

\begin{multline*}
  \J(\Piez)=\left\{[\1],[\ZZ_2],[\ZZ_3],[\DD_2^z],[\DD_3^z],[\ZZ_2^{-}],[\ZZ_4^-],[\DD_2],[\DD_3],[\DD_4^d],[\DD_6^d],\right.\\
  \left.[\SO(2)],[\OO(2)],[\OO(2)^-],[\octa^-],[\OO(3)]\right\}.
\end{multline*}

The symmetry classes for $\OO(3)$ representation on $\Ela$ and $\Sym$ are the same as the symmetry classes for $\SO(3)$ representation which can be found in \cite{forte1997symmetry} and \cite{OKDD2018a}, except that each type I subgroup occurring in the list of isotropy classes has to be replaced by the corresponding type II subgroup (see \autoref{sec:O3-subgroups}). Indeed, $-\id$ acts trivially on $\Ela$ and $\Sym$, and we have
\begin{equation*}
  \J(\Ela)=\set{[\1],[\ZZ_2\oplus \ZZ_2^c],[\DD_2\oplus \ZZ_2^c],[\DD_3\oplus \ZZ_2^c],[\DD_4\oplus \ZZ_2^c],[\octa\oplus \ZZ_2^c],[\OO(2)\oplus \ZZ_2^c],[\SO(3)\oplus \ZZ_2^c]},
\end{equation*}
and
\begin{equation*}
  \J(\Sym)=\set{[\DD_2\oplus \ZZ_2^c],[\OO(2)\oplus\ZZ_2^c],[\SO(3)\oplus \ZZ_2^c]}.
\end{equation*}

We deduce the isotropy classes of the Piezoelectricity law $\mathcal{P}\text{iez}=\Ela \oplus \Piez \oplus \Sym$ from lemma \ref{lem:direct_sum} by calculating the clips operations between the isotropy classes of $\Ela$, $\Piez$ and $\Sym$ (see \autoref{tab:Clips} and \autoref{tab:Clips2} for clips between type II and III $\OO(3)$-subgroups and \cite[table 1]{Olive2019} for clips between two type I and remark 3.1 (\ref{rem:2typeII}) for type I with type II and two type II).

\begin{thm}\label{lem:J(Piez)}
  There exists 25 isotropy classes for the Piezoelectricity law $\Ela \oplus \Piez \oplus \Sym$ given by
  \begin{multline*}
    \J(\mathcal{P}\text{iez})=\left\{[\1],[\ZZ_2],[\ZZ_3],[\ZZ_4],[\DD_2],[\DD_3],[\DD_4],[\SO(2)],[\OO(2)],[\OO(3)],\right. \\
    \left.[\ZZ_2\oplus \ZZ_2^c],[\DD_2\oplus \ZZ_2^c],[\DD_3\oplus \ZZ_2^c], [\DD_4\oplus \ZZ_2^c],[\octa\oplus \ZZ_2^c],[\OO(2)\oplus \ZZ_2^c], \right. \\
    \left.[\ZZ_2^-],[\ZZ_4^-],[\DD_2^z],
    [\DD_3^z],[\DD_4^z],[\DD_2^d],[\DD_4^d],[\DD_6^d],[\octa^-],[\OO(2)^-]\right\}.
  \end{multline*}
\end{thm}

\appendix

\section{Closed $\OO(3)$-subgroup}
\label{sec:O3-subgroups}

Given a subgroup $\Gamma$ of $\OO(3)$, set
\begin{equation*}
  \Gamma_+=\set{g\in \Gamma;\ \det g=1} \text{ and } \Gamma_-=\set{g\in \Gamma;\ \det g=-1}.
\end{equation*}
Then $\Gamma=\Gamma_+ \cup \Gamma_-$ and we have the following classification (see \cite{GSS1988}).

\begin{description}
	\item[Type I:] A closed subgroup $\Gamma$ of $\OO(3)$ is of type I if $\Gamma_-=\emptyset$, in which case $\Gamma$ is a subgroup of $\SO(3)$. A subgroup of type I is conjugate to one in the following list
	\begin{equation*}
	\SO(3),\quad \OO(2),\quad \SO(2), \quad \DD_n, \quad \ZZ_n, \quad \tetra, \quad \octa, \quad \ico, \text{ or} \quad \mathds{1}
	\end{equation*}
	\item[Type II:] A closed subgroup $\Gamma$ of $\OO(3)$ is of type II if $-\id \in \Gamma$, in which case $-\Gamma_-=\Gamma_+$ and $\Gamma=\Gamma_+\cup -\Gamma_+$. A subgroup of type II is a direct product of a subgroup of type I and the subgroup $\ZZ_2^c=\set{\pm \id}$, it is a conjugate to one in the following list
	\begin{equation*}
	\SO(3)\oplus \ZZ_2^c,\quad \OO(2)\oplus \ZZ_2^c,\quad \SO(2)\oplus \ZZ_2^c,\quad \ZZ_m\oplus \ZZ_2^c,\quad \DD_m\oplus \ZZ_2^c,\quad \tetra\oplus \ZZ_2^c,\quad \octa\oplus \ZZ_2^c, \quad \ico\oplus \ZZ_2^c
	\end{equation*}
	\item[Type III:] A closed subgroup $\Gamma$ of $\OO(3)$ is of type III if $\Gamma_-\neq \emptyset$ but $-\id \notin \Gamma$. Then, $\Gamma_+$ is a subgroup of index 2 in the subgroup $\tilde{\Gamma}=\Gamma_+\cup (-\Gamma_-)$ of $\SO(3)$. In that case, there exists $\gamma \in \tilde{\Gamma}\setminus \Gamma_+$ such that $-\Gamma_- = \gamma \Gamma_+$ and $\Gamma=\Gamma_+\cup -\gamma \Gamma_+$. Five representatives of conjugacy classes of subgroups of type III can be deduced
	\begin{align} \label{typeIII}
	& \ZZ_{2n}^-=\ZZ_n \cup (-\vR(\ee_3,\frac{\pi}{n})) \ZZ_n \quad \forall n \geq 1. \\
	& \Dnz=\ZZ_n \cup (-\vR(e_1,\pi))\ZZ_n  \quad \forall n\geq2.                     \\
	& \Dnd=\DD_n \cup (-\vR(\ee_3,\frac{\pi}{n})) \DD_n \quad \forall n \geq 1.       \\
	& \octa^-=\tetra \cup (-\vR(\ee_3,\frac{\pi}{2})) \tetra.                         \\
	& \OO(2)^-=\SO(2)\cup (-\vR(\ee_1,\pi)) \SO(2).
	\end{align}
\end{description}

\begin{exmp}\label{exmp:1}
  $    \ZZ_4^-  =\ZZ_2\cup\left(-\vR(\ee_3,\pi/2)\ZZ_2\right) =\set{\id,\vR(\ee_3,\pi),-\vR(\ee_3,\pi/2),-\vR(\ee_3,3 \pi/2)}
  $
\end{exmp}
We propose in Table~\ref{tab:TabGen} generators of type I and type III closed $\OO(3)$-subgroups.

\begin{table}[h]
  \renewcommand{\arraystretch}{2}
  \begin{center}
    \begin{tabular}{|>{$}l<{$}|>{$}c<{$}|>{$}c<{$}|}
      \hline
      \displaystyle \text{Type I subgroup}   & \text{Order} & \displaystyle \text{Generators}                                                                                                     \\ \hline
      \ZZ_{n},\ n\geq2                       & n            & \vR\left(\ee_{3},\sfrac{2\pi}{n}\right)                                                                                             \\ \hline
      \displaystyle \DD_{n},\ n\geq2         & 2n           & \vR\left(\ee_{3},\sfrac{2\pi}{n}\right),\ \vR(\ee_{1},\pi)                                                                          \\ \hline
      \displaystyle \tetra                   & 12           & \displaystyle \vR(\ee_{3},\pi),\ \vR(\ee_{1},\pi),\ \vR(\ee_{1}+\ee_{2}+\ee_{3},\sfrac{2\pi}{3})                                    \\ \hline
      \displaystyle \octa                    & 24           & \displaystyle \vR(\ee_{3},\sfrac{\pi}{2}),\ \vR(\ee_{1},\pi),\ \vR(\ee_{1}+\ee_{2}+\ee_{3},\sfrac{2\pi}{3})                         \\ \hline
      \displaystyle \ico                     & 60           & \vR(\ee_{3},\pi),\ \vR(\ee_{1}+\ee_{2}+\ee_{3},\sfrac{2\pi}{3}), \vR(\ee_{1}+\phi\ee_{3},\sfrac{2\pi}{5})\quad \phi:=(1+\sqrt{5})/2 \\ \hline
      \displaystyle \SO(2)                   & \infty       & \vR(\ee_{3},\theta),\quad \theta\in [0,2\pi]                                                                                        \\ \hline
      \displaystyle \OO(2)                   & \infty       & \SO(2),\vR(\ee_{1},\pi)                                                                                                             \\ \hline \hline
      \displaystyle \text{Type III subgroup} &              &                                                                                                                                     \\ \hline
      \displaystyle \ZZ^-_{2}                & 2            & \displaystyle -\vR(\ee_{3},\pi)                                                                                                     \\ \hline
      \displaystyle \ZZ^-_{2n},\ n\geq2      & 2n           & \displaystyle -\vR\left(\ee_{3},\sfrac{\pi}{n}\right)                                                                               \\ \hline
      \displaystyle \Dnd,\ n\geq2            & 4n           & \displaystyle -\vR\left(\ee_{3},\sfrac{\pi}{n}\right),\ \vR(\ee_{1},\pi)                                                            \\ \hline
      \displaystyle \Dnz,\ n\geq2            & 2n           & \displaystyle  \vR\left(\ee_{3},\sfrac{2\pi}{n}\right),\ -\vR(\ee_{1},\pi)                                                          \\ \hline
      \displaystyle \octa^-                  & 24           & \displaystyle -\vR(\ee_{3},\sfrac{\pi}{2}),\ -\vR(\ee_{2}-\ee_{3},\pi)                                                              \\ \hline
      \displaystyle \OO(2)^-                 & \infty       & \SO(2),-\vR(\ee_{1},\pi)                                                                                                            \\ \hline
    \end{tabular}
  \end{center}
  \caption{Generators of closed $\OO(3)$-subgroups}
  \label{tab:TabGen}
\end{table}

We give now some useful decompositions of subgroups $\tetra,\octa$ and $\ico$. To do so, first introduce subgroups
\begin{equation}\label{eq:AxesZnDn}
  \ZZ_n^{\uu}:=\left\langle \vR\left(\uu,\frac{2\pi}{n}\right)\right\rangle,\quad \DD_n^{\uu,\vv}:=\left\langle \vR\left(\uu,\frac{2\pi}{n}\right),\vR(\vv,\pi)\right\rangle
\end{equation}
where $\uu,\vv$ are a non--zero orthogonal vectors. In such notations, the axis $\langle\uu\rangle$ generated by $\uu$ is said to be the primary axis of $\ZZ_n^{\uu}$ and $\DD_n^{\uu,\vv}$, while $\langle\vv\rangle$ is said to be a secondary axis of $\DD_n^{\uu,\vv}$. In fact we have
\begin{equation*}
  \ZZ_n=\ZZ_n^{\ee_3},\quad \DD_n=\DD_n^{\ee_3,\ee_1}
\end{equation*}
and for any $g\in \SO(3)$
\begin{equation*}
  g\ZZ_ng^{-1}=\ZZ_n^{g\ee_3},\quad g\DD_ng^{-1}=\DD_n^{g\ee_3,g\ee_1}.
\end{equation*}
In particular, $\ZZ_n$ is given by
\begin{equation}\label{eq:Zn}
\ZZ_n=\set{\vR\left(\ee_3,\frac{2k\pi}{n}\right),k=0,\dotsc,n-1}
\end{equation}
And $\DD_n$ by
\begin{equation}\label{eq:Dn}
\DD_n=\set{\vR\left(\ee_3,\frac{2k\pi}{n}\right),\vR(\bb_{i},\pi);\quad k=0,\dotsc,n-1,\ i=1,\dotsc,n}
\end{equation}
with $\bb_i$ being the secondary axis of $\DD_n$ in the $(xy)$-plane such that $\bb_1=\ee_1$ and $\bb_i=\vR\left(\ee_3,\frac{\pi}{n}\right)\bb_{i-1}$, for $i\geq2$.
We propose now details on compositions of the subgroups $\tetra$, $\octa$ and $\ico$, with explicit axes for all cyclic and dihedral subgroups they contain.

First we have (see~\cite{Ihrig1984} for instance)
\begin{equation}\label{eq:Decomposition_tetra}
  \tetra=\bigcup_{i=1}^3\ZZ_2^{\ee_i}\cup\bigcup_{j=1}^4 \ZZ_3^{\pmb{s}_{t_j}}
\end{equation}
with
\begin{equation*}
  \pmb{s}_{t_1}=\ee_1+\ee_2+\ee_3,\quad \pmb{s}_{t_2}=\ee_1-\ee_2-\ee_3,\quad \pmb{s}_{t_3}=-\ee_1+\ee_2-\ee_3,\quad \pmb{s}_{t_4}=-\ee_1-\ee_2+\ee_3.
\end{equation*}
For the cubic group we have
\begin{equation}\label{eq:Decomposition_Cube}
  \octa=\bigcup_{i=1}^3 \ZZ_4^{\ee_i}\cup\bigcup_{j=1}^4 \ZZ_3^{\pmb{s}_{t_j}}\cup\bigcup_{k=1}^6\ZZ_2^{\pmb{a}_{c_k}}
\end{equation}
with vectors $\ba_k$ given by
\begin{align}\label{eq:Vecteurs_ak_Ordre2_Cube}
  \ba_{c_1} & =\be_1+\be_2,\quad \ba_{c_2}=\be_1-\be_2,\quad \ba_{c_3}=\be_1+\be_3  \\ \notag
  \ba_{c_4} & =\be_1-\be_3,\quad \ba_{c_5}=\be_2+\be_3,\quad \ba_{c_6}=\be_2-\be_3.
\end{align}
Finally we have
\begin{equation}\label{eq:Decomposition_Ico}
  \ico=\bigcup_{i=1}^6 \ZZ_5^{\uu_i}\cup\bigcup_{j=1}^{10} \ZZ_3^{\vv_j}\cup\bigcup_{k=1}^{15}\ZZ_2^{\pmb{w}_k}
\end{equation}
Taking $\phi:=(1+\sqrt{5})/5$ to be the golden ratio, vectors $\uu_i$ are obtained as centers of icosahedron faces (~\cite[Figure 11]{Olive2019}):
\begin{align*}
  \uu_1 & :=(1+3\phi)\ee_1+(2+\phi)\ee_3,\quad \uu_2:=(2+\phi)\ee_1+(1+3\phi)\ee_2,\quad \uu_3:=(2+\phi)\ee_2-(1+3\phi)\ee_3  \\
  \uu_4 & :=-(2+\phi)\ee_2-(1+3\phi)\ee_3,\quad \uu_5=(1+3\phi)\ee_1-(2+\phi)\ee_3,\quad \uu_6:=(2+\phi)\ee_1-(1+3\phi)\ee_2.
\end{align*}
Then vectors $\vv_j$ are obtained from vertices of icosahedron:
\begin{align*}
  \vv_1 & :=\ee_1+\ee_2+\ee_3,\quad \vv_2:=\phi\ee_1+\frac{1}{\phi}\ee_2,\quad \vv_3:=\ee_1+\ee_2-\ee_3,\quad \vv_4:=\phi\ee_2-\frac{1}{\phi}\ee_3    \\
  \vv_5 & :=\phi\ee_2+\frac{1}{\phi}\ee_3,\quad \vv_6:=-\ee_1+\ee_2+\ee_3,\quad \vv_7:=-\phi\ee_1+\frac{1}{\phi}\ee_2,\quad \vv_8:=-\ee_1+\ee_2-\ee_3 \\
  \vv_9 & :=-\frac{1}{\phi}\ee_1-\phi\ee_3,\quad \vv_{10}:=\frac{1}{\phi}\ee_1-\phi\ee_3
\end{align*}
while vectors $\pmb{w}_k$ are obtained from its edges:
\begin{align*}
  \pmb{w}_1    & :=\ee_1+(\phi+1)\ee_2+\left(1+\frac{1}{\phi}\right)\ee_3,\quad \pmb{w}_2:=\left(1+\frac{1}{\phi}\right)\ee_1+\ee_2+(\phi+1)\ee_3
  \\
  \pmb{w}_3    & :=(\phi+1)\ee_1+\left(1+\frac{1}{\phi}\right)\ee_2+\ee_3,\quad \pmb{w}_4:=\phi\ee_1,\quad \pmb{w}_5:=(\phi+1)\ee_1+\left(1+\frac{1}{\phi}\right)\ee_2-\ee_3 \\
  \pmb{w}_6    & :=(\phi+1)\ee_1-\left(1+\frac{1}{\phi}\right)\ee_2+\ee_3,\quad
  \pmb{w}_7:=\left(1+\frac{1}{\phi}\right)\ee_1-\ee_2+(\phi+1)\ee_3,\quad                                                                                                    \\
  \pmb{w}_8    & :=(\phi+1)\ee_1-\left(1+\frac{1}{\phi}\right)\ee_2-\ee_3,\quad
  \pmb{w}_9:=\ee_1+(\phi+1)\ee_2-\left(1+\frac{1}{\phi}\right)\ee_3                                                                                                          \\
  \pmb{w}_{10} & :=\phi\ee_2,\quad \pmb{w}_{11}:=-\ee_1+(\phi+1)\ee_2+\left(1+\frac{1}{\phi}\right)\ee_3, \quad
  \pmb{w}_{12}=\phi\ee_3                                                                                                                                                     \\
  \pmb{w}_{13} & :=-\left(1+\frac{1}{\phi}\right)\ee_1+\ee_2+(\phi+1)\ee_3,\quad
  \pmb{w}_{14}:=-\ee_1+(\phi+1)\ee_2-\left(1+\frac{1}{\phi}\right)\ee_3,                                                                                                     \\
  \pmb{w}_{15} & :=-\left(1+\frac{1}{\phi}\right)\ee_1-\ee_2+(\phi+1)\ee_3.
\end{align*}
Finally, the subgroups $\SO(2)$ and $\OO(2)$ are given by
\begin{equation}\label{eq:SO(2)}
\SO(2)=\set{\vR(\ee_3,\theta),\theta\in[0,2\pi]}
\end{equation}
and 
\begin{equation}\label{eq:O(2)}
\OO(2)=\set{\vR(\ee_3,\theta),\vR(\bb,\pi),\  \theta\in[0,2\pi],\ \bb\in (xy)-\text{plane}}.
\end{equation}
\bibliographystyle{abbrv}
\bibliography{piezorefs}
\end{document}